\pdfoutput=1

\documentclass[12pt]{amsart}
\usepackage[dvips]{graphicx}
\usepackage{color}
\usepackage{amssymb} 
\usepackage[english]{babel}
\usepackage{amsmath}
\usepackage{amsfonts}
\usepackage{latexsym}
\usepackage{amscd}
\usepackage[latin1]{inputenc}
\usepackage{enumerate}
\usepackage{afterpage}
\def\co{\colon\thinspace} 
\def\a{\alpha} 
\def\b{\beta} 
\def\d{\delta} 
\def\e{\epsilon} 
\def\g{\gamma} 
\def\l{\lambda} 
\def\i{{\bf i}} 
 
\def\s{\sigma} 
\def\t{\tau} 
\def\th{\theta} 
\def\cal{\mathcal} 
\def\CC{\mathcal C}

\def\M{{\cal M}} 
\def\S{{\cal S}} 
\def\DD{{\cal D}} 
\def\V{{\cal V}} 
\def\W{{\cal W}}

\def\P{{\cal P}} 
\def\PC{\mathcal {PC}} 
\def\RR{{\cal R}} 
\def\T{{\cal T}} 
\def\Th{{\rm Th}} 
\newcommand{\C}{\mathbb{C}} 
\newcommand{\A}{\mathbb{A}} 
 
\newcommand{\R}{\mathbb{R}} 
\newcommand{\NN}{\mathbb{N}} 
\newcommand{\Z}{\mathbb{Z}} 
\newcommand{\Q}{\mathbb{Q}}

\newcommand{\D}{\mathbb{D}} 
\newcommand{\HH}{\mathbb{H}} 
\def\tr{\mathop{\rm Tr}} 
\def\Tr{\mathop{\rm Tr}}

\def\PSL{\mathop{\rm PSL}} 
\def\CH{\mathop{\rm CH}} 
\def\Id{\mathop{\rm Id}} 
\renewcommand{\to}{\longrightarrow} 
\newcommand{\dd}{\partial}

\def \pl{\mathop{\rm{pl}}} 
\def\ML{\mathop{{ ML}} } 
\def\MLQ{\mathop{{ ML_{\Q}}} } 
\def\PML{\mathop{{PML}}}

\def\teich{{\cal T}} 
\def\Teich{Teichm\"uller } 
\def\ttau{{\underline{\tau}}} 
\def\tw{\rm{tw}}

\def\varnothing{\emptyset}

\newtheorem{Theorem}{Theorem}[section] 
\newtheorem{Lemma}[Theorem]{Lemma} 
\newtheorem{Proposition}[Theorem]{Proposition} 
\newtheorem{Corollary}[Theorem]{Corollary} 
\newtheorem{introthm}{Theorem}

 \theoremstyle{definition} 
\newtheorem{Definition}[Theorem]{Definition} 

\newtheorem{Notation}[Theorem]{Notation}

\theoremstyle{remark}
\newtheorem{Remark}[Theorem]{Remark}

\numberwithin{equation}{section}

\begin{document}

\title[The asymptotic directions]{The asymptotic directions of pleating rays in the Maskit embedding.}

\begin{abstract}
 
This article was born as a generalisation of the analysis made by Series in~\cite{Series1} where she made the first attempt to plot a deformation space of Kleinian group of more than 1 complex dimension. We use the Top Terms' Relationship proved by the author and Series in~\cite{Maloni-Series} to determine the asymptotic directions of pleating rays in the Maskit embedding of a hyperbolic surface $\Sigma$ as the bending measure of the `top' surface in the boundary of the convex core tends to zero. The Maskit embedding $\cal M$ of a surface $\Sigma$ is the space of geometrically finite groups on the boundary of quasifuchsian space for which the `top' end is homeomorphic to $\Sigma$, while the `bottom' end consists of triply punctured spheres, the remains of $\Sigma$ when the pants curves have been pinched. Given a projective measured lamination $[\eta]$ on $\Sigma$, the \emph{pleating ray} $\P = \P_{[\eta]}$ is the set of groups in $\M$ for which the bending measure $\pl^+(G)$ of the top component $\dd \CC^+$ of the boundary of the convex core of the associated $3$-manifold $\HH^{3}/G$ is in the class $[\eta]$.

\end{abstract}

\author{Sara Maloni}
\address{Mathematics Institute, 
University of Warwick \\
Coventry CV4 7AL\\ UK}
\email{S.Maloni@warwick.ac.uk}

\date{\today}
\maketitle

\section{Introduction}\label{sec:introduction}

Let $\Sigma$ be a surface of negative Euler characteristic together with a pants decomposition $\P$. Kra's plumbing construction endows $\Sigma$ with a projective structure as follows. Replace each pair of pants by a triply punctured sphere and glue, or `plumb', adjacent pants by gluing punctured disk neighbourhoods of the punctures. The gluing across the $i^{th}$ pants curve is defined by a complex parameter $\tau_i \in \C$. The associated holonomy representation $\rho \co \pi_1(\Sigma) \to PSL(2,\C)$ gives a projective structure on $\Sigma$ which depends holomorphically on the $\tau_i$. In particular, the traces of all elements $\rho(\g) , \g \in \pi_1(\Sigma)$, are polynomials in the $\tau_i$.

In \cite{Maloni-Series} the author and Series proved a formula, called \textit{Top Terms' Relationship}, which is Theorem \ref{top} in Section \ref{ssub2.1}, giving a simple linear relationship between the coefficients of the top terms of $\rho(\g)$, as polynomials in the $\tau_i$, and the Dehn--Thurston coordinates of $\g$ relative to $\P$, see Section \ref{ssub2.1} for the definitions. This result generalises the previous results proved by Keen and Series in \cite{Keen-Series1} in the case of the once punctured torus $\Sigma_{1,1}$ and by Series in \cite{Series1} for the twice punctured torus $\Sigma_{1,2}$. These formulas were used in the case $\Sigma = \Sigma_{1,1}, \Sigma_{1,2}$ to determine the asymptotic directions of pleating rays in the Maskit embedding of $\Sigma$ as the bending measure of the `top' surface in the boundary of the convex core tends to zero, see Section \ref{sec:background} for the definitions. In the present article we will use the general Top Terms' Relationship to generalise the description of asymptotic directions of pleating rays to the case of an arbitrary hyperbolic surface $\Sigma$, see Theorem \ref{TA} in Section \ref{sec3}. The Maskit embedding $\cal M$ of a surface $\Sigma$ is the space of geometrically finite groups on the boundary of quasifuchsian space for which the `top' end is homeomorphic to $\Sigma$, while the `bottom' end consists of triply punctured spheres, the remains of $\Sigma$ when the pants curves have been pinched. As such representations vary in the character variety, the conformal structure on the top side varies over the Teichm\"uller space $\T(\Sigma)$, see Section \ref{sub:maskit} for a detailed discussion.

Let $\Sigma=\Sigma_{g,b}$, and suppose we have a geometrically finite free and discrete representation $\rho$ for which $M_{\rho} = \Sigma \times \R$. Denote $\xi =\xi(\Sigma) = 3g-3+b$ the complexity of the surface $\Sigma$. Fix disjoint, non-trivial, non-peripheral and non-homotopic simple closed curves $\s_1,\ldots, \s_{\xi}$ which form a maximal pants decomposition of $\Sigma$. We consider groups for which the conformal end $\omega^-$ is a union of triply punctured spheres glued across punctures corresponding to $\s_1,\ldots, \s_{\xi}$, while $\omega^+$ is a marked Riemann surface homeomorphic to $\Sigma$. Kra's plumbing construction gives us an explicit parametrisation of a holomorphic family of representation $\rho_{\ttau}\co \pi_{1}(\Sigma) \to G(\ttau) \in \PSL(2,\C)$ such that, for certain values $\ttau = (\tau_1,\ldots, \tau_{\xi}) \in \C^{\xi}$ of the parameters, $\rho_{\ttau}$ has the above geometry, see Section \ref{sub2.2} for the definition of this construction. The Maskit embedding is the map which sends a point $X \in \teich(\Sigma)$ to the point $\ttau = (\tau_1,\ldots, \tau_{\xi}) \in \C^{\xi}$ for which the group $G(\ttau)$ has $\omega^+ = X$. Denote the image of this map by $\M = \M(\Sigma)$. Note that, with abuse of notation, we will also call \textit{Maskit embedding} the image $\M$ of the map $\teich(\Sigma) \to \C^{\xi}$ just described.

We investigate $\cal M$ using the method of pleating rays. Given a projective measured lamination $[\eta]$ on $\Sigma$, the \emph{pleating ray} $\P = \P_{[\eta]}$ is the set of groups in $\M$ for which the bending measure $\pl^+(G)$ of the top component $\dd \CC^+$ of the boundary of the convex core of the associated $3$-manifold $\HH^{3}/G$ is in the class $[\eta]$. It is known that $\cal P$ is a real 1-submanifold of $\cal M$. In fact we can parametrise this ray by $\th \in (0, c_\eta)$, where $c_\eta \in (0, \pi)$, so that we associate to $\th$ the group $G_\th \in \cal P$ such that $\pl^+(G_\th) = \th\eta$, see Theorem 6 in \cite{Series3} for the case $\Sigma = \Sigma_{1,1}$. Note that this result relies on Thurston's bending conjecture which is solved for rational lamination by work of Otal and Bonahon and in the case of punctured tori by work of Series. For a  general (irrational) lamination, anyway, we can only conjecture that the real dimension of the associated pleating ray is 1. Our main result is a formula for the asymptotic direction of $\cal P$ in $\cal M$ as the bending measure tends to zero, in terms of natural parameters for the representation space $\cal R$ and the Dehn--Thurston coordinates of the support curves to $[\eta]$ relative to the pinched curves on the bottom side. This leads to a method of locating $\cal M$ in $\cal R$. 

We restrict to pleating rays for which $[\eta]$ is \emph{rational}, that is, supported on closed curves, and for simplicity write $\P_{\eta}$ in place of $\P_{[\eta]}$, although noting that $\P_{\eta}$ depends only on $[\eta]$. From general results of Bonahon and Otal~\cite{Bon-Otal}, for any pants decomposition $\g_1,\ldots,\g_{\xi}$ such that $\s_1,\ldots, \s_{\xi}, \g_1,\ldots, \g_{\xi}$ are mutually non-homotopic and fill up $\Sigma$ (see section \ref{ssub:manifolds} for the definitions), and any pair of angles $\theta_i \in (0,\pi)$, there is a unique group in $\M$ for which the bending measure of $\dd \CC^+$ is $\sum_{i = 1}^{\xi} \theta_i \delta_{\g_i}$. (This extends to the case $\th_i = 0$ for $i \in I \subset \{1,\ldots,\xi\}$ provided $\{\s_1,\ldots, \s_{\xi}, \g_j | j \notin I\}$ fill up $\Sigma$ and also to the case $\th = \pi$.) Thus given $\eta =\sum_{i = 1}^{\xi} a_i \delta_{\g_i}$, there is a unique group $G = G_{\eta}(\th) \in \M$ with bending measure $\pl^+(G)=\theta \eta$ for any sufficiently small $\theta >0$.

Let $\S$ denote the set of homotopy classes of multiple loops on $\Sigma$, and let the pants curves defining $\P$ be $\s_i, i=1,\ldots,\xi$. The \emph{Dehn--Thurston coordinates} of $\g \in \S$ are $\i(\g) = (q_1,p_1, \ldots,q_{\xi},p_{\xi})$, where $q_{i} = i(\g,\s_{i}) \in \NN \cup \{0\}$ is the geometric intersection number between $\g$ and $\s_{i}$ and $p_{i} \in \Z$ is the twist of $\g$ about $\s_i$. For a detailed discussion about this parametrisation see Section \ref{ssub2.1} below or Section 3 of \cite{Maloni-Series}. If $\eta = \sum_{i = 1}^{\xi} a_i \delta_{\g_i}$, the above condition of Bonahon and Otal on $\s_1,\ldots, \s_{\xi}, \g_1,\ldots, \g_{\xi}$ is equivalent to ask $q_i(\eta) >0, \forall i=1,\ldots, \xi$. We call such laminations \emph{admissible}. 

The main result of this paper is the following. We will state this result more precisely, as Theorem \ref{TA} in Section 3.

\begin{introthm}\label{thmA} 
    Suppose that $\eta = \sum_{i = 1}^{\xi} a_i \delta_{\g_i}$ is admissible. Then, as the bending measure $pl^+(G) \in [\eta]$ tends to zero, the pleating ray $\cal P_{\eta}$ approaches the line $$ \Re \tau_i = \frac{p_i(\eta)}{q_i(\eta)},\;\;\;\;\; \frac{\Im \tau_1}{\Im \tau_j} = \frac{q_j(\eta)}{q_1(\eta)}.$$ 
\end{introthm}

We should note that, in contrast to Series' statement, we were able to dispense with the hypothesis `$\eta$ non exceptional' (see \cite{Series1} for the definition), because we were able to improve the original proof. In addition, the definition of the line is different because we have corrected a misprint in \cite{Series1}.

One might also ask for the limit of the hyperbolic structure on $\dd \CC^+(G)$ as the bending measure tends to zero. The following result is an immediate consequence of the first part of the proof of Theorem~\ref{thmA}.

\begin{introthm}\label{thmC} 
    Let $\eta = \sum_{1}^{\xi}a_i \delta_{\g_i}$ be as above. Then, as the bending measure $\pl^+(G) \in [\eta]$ tends to zero, the induced hyperbolic structure of $\dd \CC^+$ along $\cal P_{\xi}$ converges to the barycentre of the laminations $\s_1, \ldots, \s_{\xi}$ in the Thurston boundary of $\teich(\Sigma)$. 
\end{introthm}

This should be compared with the result in~\cite{Series2}, that the analogous limit through groups whose bending laminations on the two sides of the convex hull boundary are in the classes of a fixed pair of laminations $[\xi^{\pm}]$, is a Fuchsian group on the line of minima of $[\xi^{\pm}]$. It can also be compared with Theorem 1.1 and 1.2 in \cite{Diaz-Series}.

Finally, we wanted to underline that the result achieved in Theorem \ref{disj} about the relationship between the Thurston's symplectic form and the Dehn--Thurston coordinated for the curves is very interesting in its own. It tells us that given two loops $\gamma, \gamma' \in \S$ which belongs to the same chart of the standard train track, see Section \ref{ssub:sympl} for the definition, then $\Omega_{\rm Th} (\gamma,\gamma') = \sum_{i=1}^{\xi} (q_i p'_i - q'_i p_i),$ where the vector ${\bf i}(\gamma)=(q_{1}, p_{1}, \ldots, q_{\xi}, p_{\xi}), {\bf i}(\gamma')=(q'_{1}, p'_{1}, \ldots, q'_{\xi}, p'_{\xi})$ are the Dehn--Thurston coordinates of the curves $\g, \g'$.

The plan of the paper is as follows. Section \ref{sec:background} provides an overview of all the background material needed for understanding and proving the main results which we will prove in Section \ref{sec3}. In particular, in Section \ref{sec:background} we will discuss issues related to curves on surfaces (for example we will recall the Dehn--Thurston coordinates of the space of measured laminations, Thurston's symplectic structure, and the curve and the marking complexes). In the same section we will also review Kra's plumbing construction which endows a surface with a projective structure whose holonomy map gives us a group in the Maskit embedding, and we will discuss the Top Terms' Relationship. Then we will recall the definition of the Maskit embedding and of the pleating rays. In Section \ref{sec3}, on the other hand, after fixing some notation, we will prove the three main results stated above. We will follows Series' method \cite{Series1}: we will state (without proof) the theorems which generalise straightforwardly to our case, but we will discuss the results which require further comments. In particular, many proofs become much more complicated when we increase the complex dimension of the parameter space from 1 or 2 to the general $\xi(\Sigma)$. It is worth noticing that, using a slightly different proof in Theorem \ref{thmA}, we were able to extend Series' result to the case of `non-exceptional' laminations, see Section \ref{sec3} for the definition. We also corrected some mistakes in the statement of Theorem \ref{thmA}.

\section{Background} \label{sec:background}

\subsection{Curves on surfaces}

\subsubsection{Dehn--Thurston coordinates} \label{ssub2.1}

In this section we review Dehn--Thurston coordinates, which extend to global coordinates for the space of measure laminations $\ML(\Sigma)$. These coordinates are effectively the same as the \textit{canonical coordinates} in~\cite{Series1}. We follow the description in~\cite{Maloni-Series}. First we need to fix some notation.

Suppose $\Sigma$ is a surface of finite type, let $\S_0 = \S_0(\Sigma)$ denote the set of free homotopy classes of connected closed simple non-boundary parallel curves on $\Sigma$, and let $\S = \S(\Sigma)$ be the set of multi-curves on $\Sigma$, that is, the set of finite unions of pairwise disjoint curves in $\S_0$. For simplicity we usually refer to elements of $\S$ as `curves' rather than `multi-curves', in other words, a curve is not required to be connected. The geometric intersection number $i(\a,\b)$ between $\a,\b \in \S$ is the least number of intersections between curves representing the two homotopy classes, that is $$i(\a, \b) = \min_{a \in \a , \; b \in \b} |a \cap b|.$$ 

Given a surface $\Sigma = \Sigma_{g}^{b}$ of finite type and negative Euler characteristic, choose a maximal set $\PC = \{ \s_{1}, \ldots, \s_{\xi}\}$ of homotopically distinct and non-boundary parallel loops in $\Sigma$ called \emph{pants curves}, where $\xi = \xi(\Sigma) = 3g-3+b$ is the complexity of the surface. These connected curves split the surface into $k=2g-2+b$ three-holed spheres $P_1,\ldots, P_k$, called \emph{pairs of pants}. (Note that the boundary of $P_i$ may include punctures of $\Sigma$.) We refer to both the set $\cal P = \{ P_1,\ldots, P_k\}$, and the set $\PC$, as a \emph{pants decomposition} of $\Sigma$.

Now suppose we are given a surface $\Sigma$ together with a pants decomposition $\PC$ as above. Given $\g \in \S$, define $q_{i} = q_{i}(\g) = i(\g,\s_{i}) \in \Z_{\geqslant 0}$ for all $i = 1, \ldots, \xi$. Notice that if $\s_{i_{1}}, \s_{i_{2}}, \s_{i_{3}}$ are pants curves which together bound a pair of pants whose interior is embedded in $\Sigma$, then the sum $q_{i_{1}} + q_{i_{2}} + q_{i_{3}}$ of the corresponding intersection numbers is even. The $q_{i}$ are usually called the \textit{length parameters} of $\g$.

To define the \textit{twist parameter} $\tw_i = \tw_i(\g) \in \Z$ of $\g$ about $\s_i$, we first have to fix a marking on $\Sigma$. (See D. Thurston's preprint~\cite{Dylan} for a detailed discussion about three different, but equivalent ways of fixing a marking on $\Sigma$.) A way of specifying the marking is by choosing a set of curves $D_i$, each one dual to a pants curve $\s_i$, see next paragraph for the definition. Then, after isotoping $\g$ into a well-defined standard position relative to $\P$ and to the marking, the twist $\tw_i$ is the signed number of times that $\g$ intersects a short arc transverse to $\s_i$. We make the convention that if $i(\g,\s_i) = 0$, then $\tw_i(\g) \geqslant 0$ is the number of components in $\g$ freely homotopic to $\s_i$.

Each pants curve $\s$ is the common boundary of one or two pairs of pants whose union we refer to as the \textit{modular surface} associated to $\s$, denoted $M(\s)$. Nota that if $\s$ is adjacent to exactly one pair of pants, $M(\s)$ is a one holed torus, while if $\s$ is adjacent to two distinct pairs of pants, $M(\s)$ is a four holed sphere. A curve $D$ is \textbf{dual} to the pants curve $\s$ if it intersect $\s$ minimally and is completely contained in the modular surface $M(\s)$. 

\begin{Remark}[Convention on dual curves]\label{dual}
    We shall need to consider \emph{dual curves} to $\s_i \in \PC$. The intersection number of such a connected curve with $\s_i$ is $1$  if $M(\s_i)$ a  one-holed torus and $2$ if it is a four-holed sphere. We adopt a useful convention introduced in~\cite{Dylan} which simplifies the formulae, in such a way as to avoid the need to distinguish between these two cases. Namely, for those $\s_i$  for which $M(\s_i)$ is $\Sigma_{1,1}$, we define the dual curve $D_i \in \S$  to be \emph{two} parallel copies of the connected curve intersecting $\s_i$ once, while if $M(\s_i)$ is $\Sigma_{0,4}$ we take a single copy. In this way we always have, by definition, $i(\s_i, D_i) =2$. See Section 2 of \cite{Maloni-Series} for a deeper discussion.
\end{Remark}

There are various ways of defining the \emph{standard position} of $\g$, leading to differing definitions of the twist. In this paper we will always use the one defined by D. Thurston~\cite{Dylan} (which we will denote $p_i(\g)$), but we refer to our previous article~\cite{Maloni-Series} for a further discussion about the different definitions of the twist parameter and for the precise relationship between them (Theorem 3.5 \cite{Maloni-Series}). With either definition, a classical theorem of Dehn~\cite{Dehn}, see also~\cite{Penner} (p 12), asserts that the length and twist parameters uniquely determine $\g \in \S$. This result was described by Dehn in a 1922 Breslau lecture~\cite{Dehn}. 

\begin{Theorem}
	\textbf{(Dehn's theorem, 1922)}\\
	Given a marking $(\PC;\DD) = (\s_{1}, \ldots, \s_{\xi};D_{1}, \ldots, D_{\xi})$ on $\Sigma$, the map ${\bf i} = {\bf i}_{(\PC;\D) } \co \S(\Sigma) \to \Z_{\geqslant 0}^{\xi} \times \Z^{\xi}$ which sends $\g \in \S(\Sigma)$ to \\
	$(q_{1}(\g), \ldots, q_{\xi}(\g);\tw_{1}(\g), \ldots, \tw_{\xi}(\g))$ is an injection. The point \\
	$(q_{1}, \ldots, q_{\xi}, \tw_{1}, \ldots, \tw_{\xi}) $ is in the image of ${\bf i}$ (and hence corresponds to a curve) if and only if: 
	\begin{enumerate}
		\renewcommand{\labelenumi}{(\roman{enumi})} 
		\item if $q_{i} = 0$, then $\tw_{i} \geqslant 0$, for each $i = 1, \ldots, \xi$. 
		\item if $\s_{i_{1}}, \s_{i_{2}}, \s_{i_{3}}$ are pants curves which together bound a pair of pants whose interior is embedded in $\Sigma$, then the sum $q_{i_{1}} + q_{i_{2}} + q_{i_{3}}$ of the corresponding intersection numbers is even. 
	\end{enumerate}
\end{Theorem}

One can think of this theorem in the following way. Suppose given a curve $\g \in \S$, whose length parameters $q_{i}(\g)$ necessarily satisfy the parity condition (ii), then the $q_{i}(\g)$ uniquely determine $\g \cap P_{j}$ for each pair of pants $P_{j}$, $j= 1, \ldots, k$, in accordance with the possible arrangements of arcs in a pair of pants, see for example~\cite{Penner}. Now given two pants adjacent along the curve $\s_{i}$, we have $q_{i}(\g)$ points of intersection coming from each side and we have only to decide how to match them together to recover $\g$. The matching takes place in the cyclic cover of an annular neighbourhood of $\s_i$. The twist parameter $tw_i(\g)$ specifies which of the $\Z$ possible choices is used for the matching.

In 1976 William Thurston rediscovered Dehn's result and extended it to a parametrisation of (Whitehead equivalence classes of) measured foliation of $\Sigma$, see Fathi, Laudenbach and Po\'enaru \cite{FLP} or Penner with Harer \cite{Penner} for a detailed discussion. Penner's approach for parametrising $\ML(\Sigma)$ is through \textit{train tracks}. Using them, Thurston also defined a symplectic form on $\ML(\Sigma)$, called \textit{Thurston's symplectic form}. Since it will be useful later, we will recall its definition and some properties in the next section.

\subsubsection{Thurston's symplectic form} \label{ssub:sympl}

We will focus on Penner's approach, following Hamenstad's notation \cite{Ham}. We will define train tracks and some other related notions, so as to be able to define the symplectic form. Then we will present an easy way to calculate it.

A \textit{train track} on the surface $\Sigma$ is an embedded 1--complex $\tau \subset \Sigma$ whose edges (called \textit{branches}) are smooth arcs with well--defined tangent vectors at the endpoints. At any vertex (called a \textit{switch}) the incident edges are mutually tangent. Through each switch there is a path of class $C^1$ which is embedded in $\tau$ and contains the switch in its interior. In particular, the branches which are incident on a fixed switch are divided into ``incoming" and ``outgoing" branches according to their inward pointing tangent vectors at the switch. Each closed curve component of $\tau$ has a unique bivalent switch, and all other switches are at least trivalent. The complementary regions of the train track have negative Euler characteristic, which means that they are different from discs with 0, 1 or 2 cusps at the boundary and different from annuli and once-punctured discs with no cusps at the boundary. A train track is called \textit{generic} if all switches are at most trivalent. Note that in the case of a trivalent vertex there is one incoming branch and two outgoing ones. 

Denote $\mathcal{B} = \mathcal{B}(\tau)$ the set of branches of $\tau$. Then a function $w\co\mathcal{B} \to \R_{\geqslant 0}$ (resp. $w\co\mathcal{B} \to \R$) is a \textit{transverse measure} (resp. \textit{weighting}) for $\tau$ if it satisfies the \textit{switch condition}, that is for all switches $v$, we want $\sum_{i} w(e_i) = \sum_{j} w(E_j)$ where the $e_i$ are the incoming branches at $v$ and $E_j$ are the outgoing ones. 
 
A train track is called \textit{recurrent} if it admits a transverse measure which is positive on every branch. A train track $\tau$ is called \textit{transversely recurrent} if every branch $b \in \mathcal{B}(\tau)$ is intersected by an embedded simple closed curve $c = c(b) \subset \Sigma$ which intersects $\tau$ transversely and is such that $\Sigma - \tau -c$ does not contain an embedded bigon, i.e. a disc with two corners on the boundary. A recurrent and transversely recurrent train track is called \textit{birecurrent}. A geodesic lamination or a train track $\lambda$ is \textit{carried by} a train track $\tau$ if there is a map $F\co \Sigma \to \Sigma$ of class $C^1$ which is isotopic to the identity and which maps $\lambda$ to $\tau$ in such a way that the restriction of its differential $dF$ to every tangent line of $\lambda$ is non--singular. A generic transversely recurrent train track which carries a complete geodesic lamination is called \textit{complete}, where we define a geodesic lamination to be complete if there is no geodesic lamination that strictly contains it.

Given a generic birecurrent train track $\tau \subset \Sigma$, we define $\V(\tau)$ to be the collection of all (not necessary nonzero) transverse measures supported on $\tau$ and let $\W(\tau)$ be the vector space of all assignments of (not necessary nonnegative) real numbers, one to each branch of $\tau$, which satisfy the switch conditions. By splitting, we can arrange $\tau$ to be generic. Since $\Sigma$ is oriented, we can distinguish the right and left hand outgoing branches, see Figure \ref{fig1}. If ${\bf n}, {\bf n'} \in \W(\tau)$ are weightings on $\tau$ (representing points in $\ML(\Sigma)$), then we denote by ${b_v}({\bf n}), {c_v}({\bf n})$ the weights of the left hand and right hand outgoing branches at $v$ respectively. The \textit{Thurston product} is defined as $$\Omega_{\rm Th} ({\bf n}, {\bf n'}) = \frac{1}{2}\sum_{v}b_v({\bf n}) c_v({\bf n'}) -b_v({\bf n'}) c_v({\bf n}).$$ 

In Theorem 3.1.4 of Penner \cite{Penner} it is proved that, if the train track $\tau \subset \Sigma$ is complete, then the interior $\stackrel{\circ}{\V}(\tau)$ of $\V(\tau)$ for a complete train track $\tau \subset \Sigma$ can be thought of as a chart on the PIL manifold $\ML_{\Q}(\Sigma)$ of rational measured laminations, that is laminations supported on multi-curves. (PIL is short for \textit{piecewise--integral--linear}, see \cite[Section 3.1]{Penner} for the definition.) In addition, in this case, we can identify $\W(\tau)$ with the tangent space to $\ML_{\Q}(\Sigma)$ at a point in $\stackrel{\circ}{\V}(\tau)$. The Thuston product $\Omega_{\rm Th}$ defined above allows us to define a symplectic structure on the PIL manifold $\ML_{\Q}(\Sigma)$. 

\begin{figure}
	\centering 
	\includegraphics [height = .9 in, viewport = 180 600 500 675]{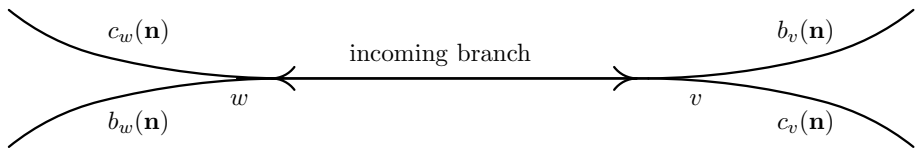} 
	\caption{Weighted branches at a switch.} 
	\label{fig1} 
\end{figure}

It is interesting to note that if $\tau$ is oriented, then there is a natural map $h_{\tau}\co\W(\tau) \to \mathrm{H}_{1}(\Sigma;\R)$, see Section 3.2 \cite{Penner}, which is related to the Thurston product by the following result. For a generalisation of this result to the case of an arbitrary (not necessarily orientable) track $\tau \in \Sigma$, see Section 3.2 \cite{Penner}.

\begin{Proposition}[Lemma 3.2.1 and 3.2.2 \cite{Penner}]
 For any train track $\tau$, $\Omega_{\rm Th} (\cdot, \cdot)$ is a skew-symmetric bilinear pairing on $\W(\tau)$. In addition, if $\tau$ is connected, oriented and recurrent, then for any ${\bf n}, {\bf n'} \in \W(\tau)$, $\Omega_{\rm Th} ({\bf n}, {\bf n'})$ is the homology intersection number of the classes $h_{\tau}({\bf n})$ and $h_{\tau}({\bf n'})$. 
\end{Proposition}

In Proposition 4.3 of \cite{Series1}, Series relates Thurston product to the Dehn--Thurston coordinates described above, but her proof works only for the case $\Sigma = \Sigma_{1,2}$, since she uses a particular choice of train tracks, called \textit{canonical train tracks}. Our idea was to use the \textit{standard train tracks}, as defined by Penner \cite{Penner} in Section 2.6. The Dehn-Thurston coordinates, using Penner's twist $\hat{p}_i$, give us a choice of a standard model and of specific weights on each edge of the track. Then one can calculate the Thurston's product, using the definition above, for a pair of curves $\gamma, \gamma' \in \S$ supported on a common standard train track. Finally, using the relationship between Penner's and D. Thurston's twist, as described by Theorem 3.5 by Maloni and Series \cite{Maloni-Series}, one can prove the following result, which will be very important in the proof of our main theorems. In particular, the standard train track are of two types: the tracks in the annuli around the pants curves and the tracks in the pair of pants. The sum of the Thurston's product in the annuli give us $\sum_{i=1}^{\xi} (q_i \hat{p}'_i - q'_i \hat{p}_i)$, using Penner's twists, while the sum of the pairs of pants give us some terms, so that the total sum give us the results that we want, that is the product $\sum_{i=1}^{\xi} (q_i p'_i - q'_i p_i)$, where we use D. Thurston's twist. We should notice that this result, although we proved it because we need it in our last section, is really interesting in its own and it is possible much more can be said from it.

\begin{Theorem}\label{disj}
   Suppose that loops $\gamma, \gamma' \in \S$ belongs to the same chart (and so are supported on a common standard train track) and they are represented by coordinates ${\bf i}(\gamma)=(q_{1}, p_{1}, \ldots, q_{\xi}, p_{\xi}), {\bf i}(\gamma')=(q'_{1}, p'_{1}, \ldots, q'_{\xi}, p'_{\xi})$. Then $\Omega_{\rm Th} (\gamma,\gamma') = \sum_{i=1}^{\xi} (q_i p'_i - q'_i p_i).$ 
   
   In addition, if $\g,\g'$ are disjoint, then $\Omega_{\rm Th} (\gamma,\g') =0$. 
\end{Theorem}

Notice that this symplectic form $\Omega_{\rm Th} (\cdot, \cdot)$ induces a map $\R^{2\xi} \to \R^{2\xi}$ defined by ${\bf x} = (x_{1},y_{1},\ldots,x_{\xi},y_{\xi}) \to {\bf x}^{*} = (y_{1}, - x_{1},\ldots,-y_{\xi},x_{\xi})$ such that $$\Omega_{\Th} (\i(\g),\i(\delta) ) = \i(\g) \cdot \i(\delta)^* $$ where $\cdot$ is the usual inner product on $\R^{2\xi}.$ To understand the meaning of the vector ${\bf x}^{*}$ better, we should recall the last Proposition of Section 3.2 of \cite{Penner} and some notation from Bonahon's work (see his survey paper \cite{Bon2} for a general introduction to the argument and for other further references). After rigorously defining the tangent space $T_{\a}\ML(\Sigma)$ with $\a \in \ML(\Sigma)$, Bonahon proved in \cite{Bon4} that we can interpret any tangent vector $v \in T_{\a}\ML(\Sigma)$ as a geodesic lamination with a transverse H\"older distribution. Note that the space $\W(\tau)$ can be seen as the space of H\"older distributions on the track $\tau,$ since it is defined to be the vector space of all assignments of not necessary nonnegative real numbers, one to each branch of $\tau$, which satisfy the switch conditions. He also characterised which geodesic laminations with transverse distributions correspond to tangent vectors to $\ML(\Sigma)$. Notice that if the lamination is carried by the track $\tau$, we can locally identify $T_{\a}\ML(\Sigma)$ with $\W(\tau)$ which is isomorphic to $\R^{2\xi}$. 

\begin{Theorem}[Theorem 3.2.4 \cite{Penner}]
    For any surface $\Sigma$, the Thurston product is a skew-symmetric, nondegenerate, bilinear pairing on the tangent space to the PIL manifold $\ML_{0}(\Sigma)$.
\end{Theorem}

\subsubsection{Complex of curves and marking complex} \label{ssub:complex}

In this section, we review the definitions of the complex of curves and of the marking complex. We will use this language in the last Section where we will prove our main Theorems. While it is not essential to use this language, we believe most readers will already be familiar with these definitions and will find easier to understand the ideas of our proofs. In addition, these tools will shorten the proofs. We summarise briefly the definition of simplicial complex and few related definitions which we will need later on, and we refer to Hatcher \cite{Hatcher} for a complete discussion. 

\begin{Definition}
     Given $K^{(0)}$ a set (of \textit{vertices}), then $K \subset \P(K^{(0)})$, where $\P(K^{(0)})$ is the power set of $K^{(0)}$, is a \textit{simplicial complex} if \begin{enumerate}
                 \item $\varnothing \notin K$;
                 \item $\forall \tau \subset \sigma \in K, \; \tau \neq \varnothing \Rightarrow \tau \in K.$ 
            \end{enumerate}
     Given $\sigma \in K$, we define the \textit{link of} $\sigma$ to be the set $\mathrm{lk}_{K}(\sigma) = \{\tau \in K | \tau \cap \sigma = \varnothing, \; \tau \cup \sigma \in K\}$.
\end{Definition}

\begin{Definition}
    Given a surface $\Sigma$, let $\CC^{(0)}(\Sigma)$ be the set of isotopy classes of essential, nonperipheral, simple closed curves in $\Sigma$. Then we define the \textit{complex of curves} $\CC(\Sigma)$ as the simplicial complex with vertex set $\CC^{(0)}(\Sigma)$ and where multicurves gives simplices. In particular $k$--simplices of $\CC(\Sigma)$ are $(k + 1)$--tuples $\{\g_0,\ldots,\g_k\}$ of distinct nontrivial free homotopy classes of simple, nonperipheral closed curves, which can be realised disjointly.
\end{Definition}

Note that this complex is obviously finite--dimensional by an Euler characteristic argument, and is typically locally infinite. If $\Sigma = \Sigma_{g,b}$, then the dimension is $\mathrm{dim}\left(\CC(\Sigma)\right) = \xi(\Sigma)-1 = 3g-4+b.$

Note that the cases of lower complexity, which are called \textit{sporadic} by Masur and Minsky \cite{MM1}, require a separate discussion. In particular if $\Sigma = \Sigma_{0,b}$ with $b \leqslant 3$, then $\CC(\Sigma)$ is empty. If $\Sigma = \Sigma_{1,0}, \Sigma_{1,1}, \Sigma_{0,4}$, using this definition, $\CC(\Sigma)$ is disconnected (in fact, it is just an infinite set of vertices). So we slightly modify the definition, in such a way that edges are placed between vertices corresponding to curves of smallest possible intersection number ($1$ for the tori, $2$ for the sphere). Finally  also in the case of an annulus, that is $\Sigma = \Sigma_{0,2} = \A$, $\CC(\Sigma)$ needs to be defined in a different way, which we do not not discuss here as it is not needed. We refer the interested to Masur and Minsky \cite{MM2} for a detailed discussion. 

We define now the marking complex. Before defining it, we need to give few additional definitions.

\begin{Definition}\label{mark}
    Given a surface $\Sigma$, a \textit{complete clean marking} $\mu$ on $\Sigma$ is a pants decomposition $\mathrm{base}(\mu) = \{\g_1, \ldots, \g_\xi\}$, called the \textit{base} of the marking, together with the choice of dual curves $D_i$ for each $i = 1, \ldots, \xi$ such that $D_i \cap \g_j = \varnothing$ for any $j \neq i$.
\end{Definition}

There are two types of \textit{elementary moves} on a complete clean marking:
\begin{enumerate}
    \item \textit{Twist}: Replace a dual curve $D_i$ by another dual curve $D'_i$ obtained from $D_i$ by a Dehn--twist or an half--twist around $\g_i$. 
    \item \textit{Flip}: Exchange a pair $(\g_i,D_i)$ with a new pair $(\g'_i,D'_i) := (D_i,\g_i)$ and change the dual curves $D_j$ with $j\neq i$ so that they will satisfy the property described in the Definition \ref{mark}. This operation is called \textit{cleaning the marking} and it is not uniquely defined.
\end{enumerate}

We will only need to use the base of the markings, so we will not describe these operations more deeply. The interested reader can refer to \cite{MM2} for a deeper analysis on this topic.

\begin{Definition}
    Given a surface $\Sigma$, let $\mathcal{MC}^{(0)}(\Sigma)$ be the set of complete clean markings in $\Sigma$. Then we define the \textit{marking complex} $\mathcal{MC}(\Sigma)$ as the simplicial complex with vertex set $\mathcal{MC}^{(0)}(\Sigma)$ and where two vertices are connected by an edge if the two markings are connected by an elementary move.
\end{Definition}

\subsection{Plumbing construction} \label{sub2.2}

In this section we review the plumbing construction which gives us the complex parameters $\t_i$ for the Maskit embedding. The idea of the plumbing construction is to manufacture $\Sigma$ by gluing triply punctured spheres across punctures. There is one triply punctured sphere for each pair of pants $P \in \P$, and the gluing across the pants curve $\s_j$ is implemented by a specific projective map depending on a parameter $\tau_j \in \C$. The $\tau_j$ will be the parameters of the resulting holonomy representation $\rho_{\ttau} \co \pi_1(\Sigma) \to PSL(2,\C)$ with $\ttau = (\tau_{1},\ldots,\tau_{\xi})$. 

More precisely, we first fix an identification of the interior of each pair of pants $P_i$ to a standard triply punctured sphere $\mathbb P$. We endow $\mathbb P$ with the projective structure coming from the unique hyperbolic metric on a triply punctured sphere. Then the gluing is carried out by deleting open punctured disk neighbourhoods of the two punctures in question and gluing horocyclic annular collars round the resulting two boundary curves, see Figure~\ref{figure2_4}. 
\begin{figure}
\centering 
\includegraphics[height=2.9cm]{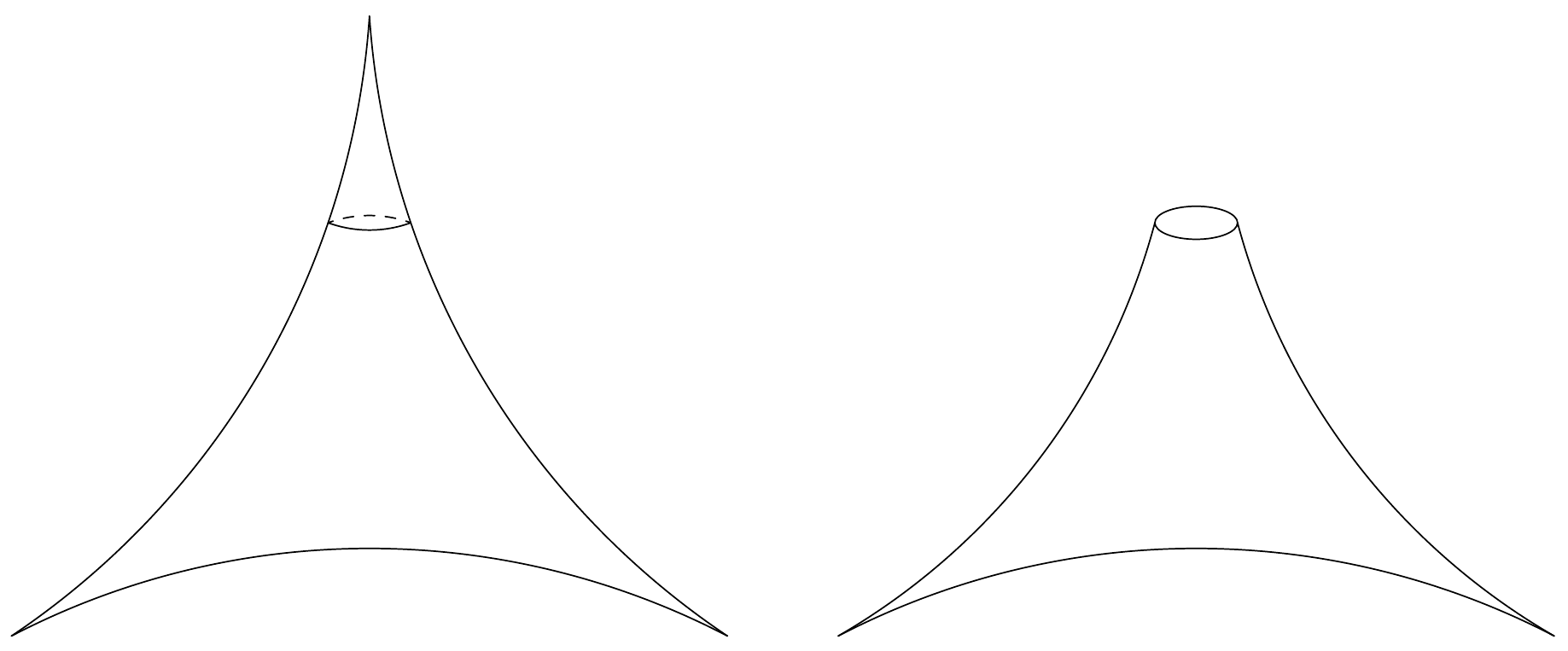} \hspace{0.4cm} 
\includegraphics[height=2.9cm]{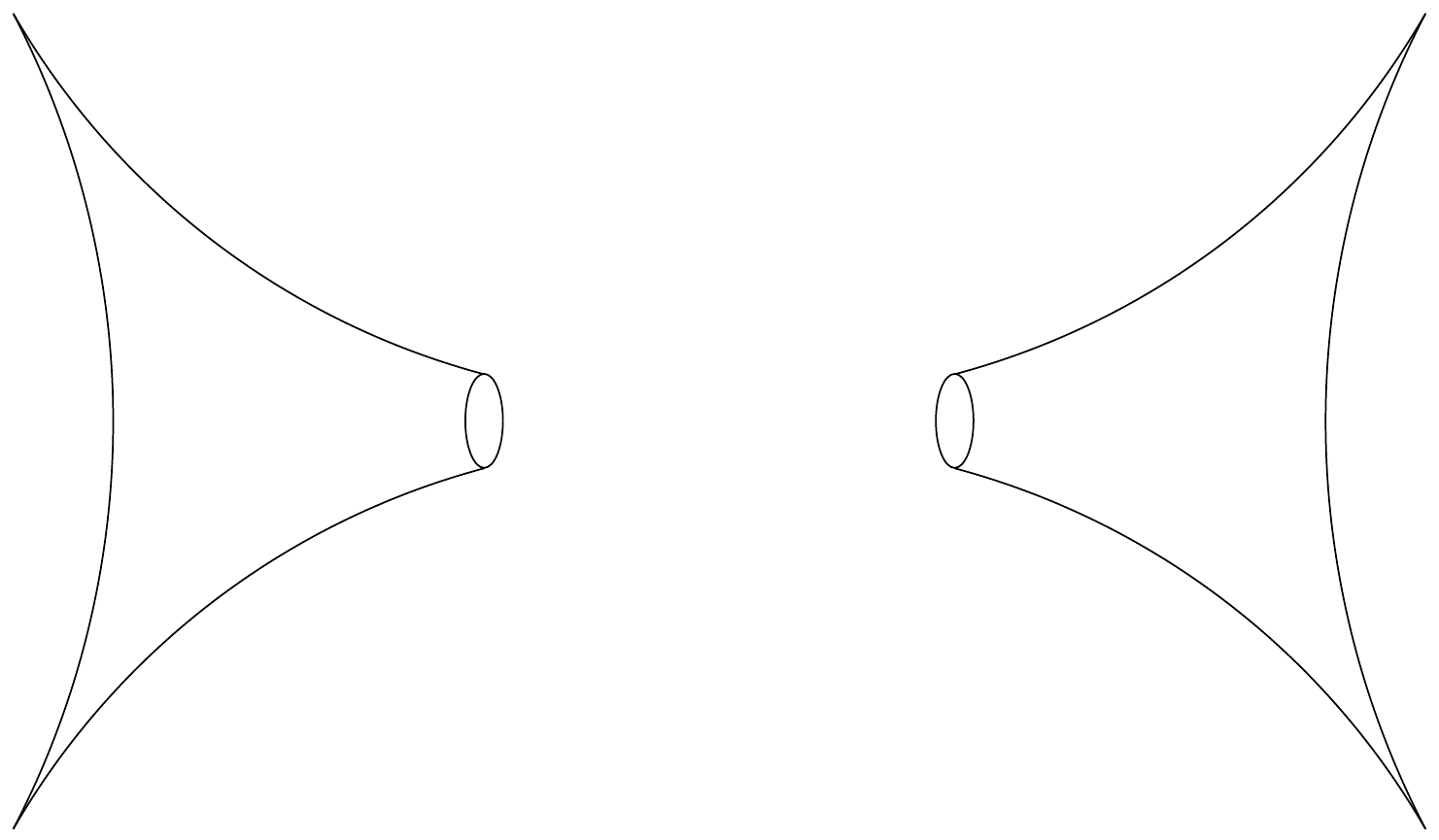} \caption{Deleting horocyclic neighbourhoods of the punctures and preparing to glue.} \label{figure2_4} 
\end{figure}

\subsubsection{The gluing} 
\label{ssub:the_gluing}

First recall (see for example \cite{Indra} p.\,207) that any triply punctured sphere is isometric to the standard triply punctured sphere $\mathbb P = \HH /\Gamma$, where $$\Gamma = \Bigl <
\begin{pmatrix}
1 & 2 \\
0 & 1 \\
\end{pmatrix}
, 
\begin{pmatrix}
1 & 0 \\
2 & 1 \\
\end{pmatrix}
\Bigr >.$$ Fix a standard fundamental domain for $\Gamma$, as shown in Figure~\ref{figure4_2}, so that the three punctures of $\mathbb P $ are naturally labelled $0,1,\infty$. Let $\Delta_0$ be the ideal triangle with vertices $ \{0,1,\infty\}$, and $\Delta_1$ be its reflection in the imaginary axis. We sometimes refer to $\Delta_0$ as the white triangle and $\Delta_1$ as the black.

\begin{figure}
\centering 
\includegraphics[height=5cm]{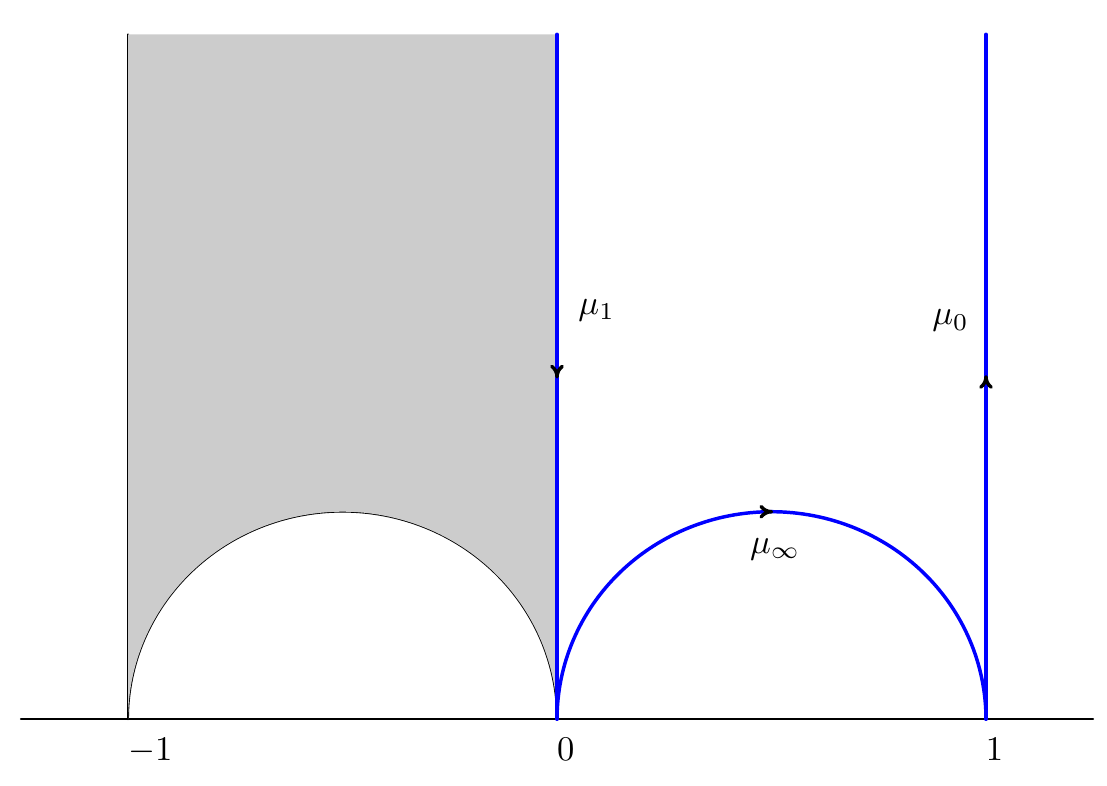} 
\caption{The standard fundamental domain for $\Gamma$. The white triangle $\Delta_0$ is unshaded.} \label{figure4_2} 
\end{figure}

\begin{figure} 
\centering 
\includegraphics[height=15cm]{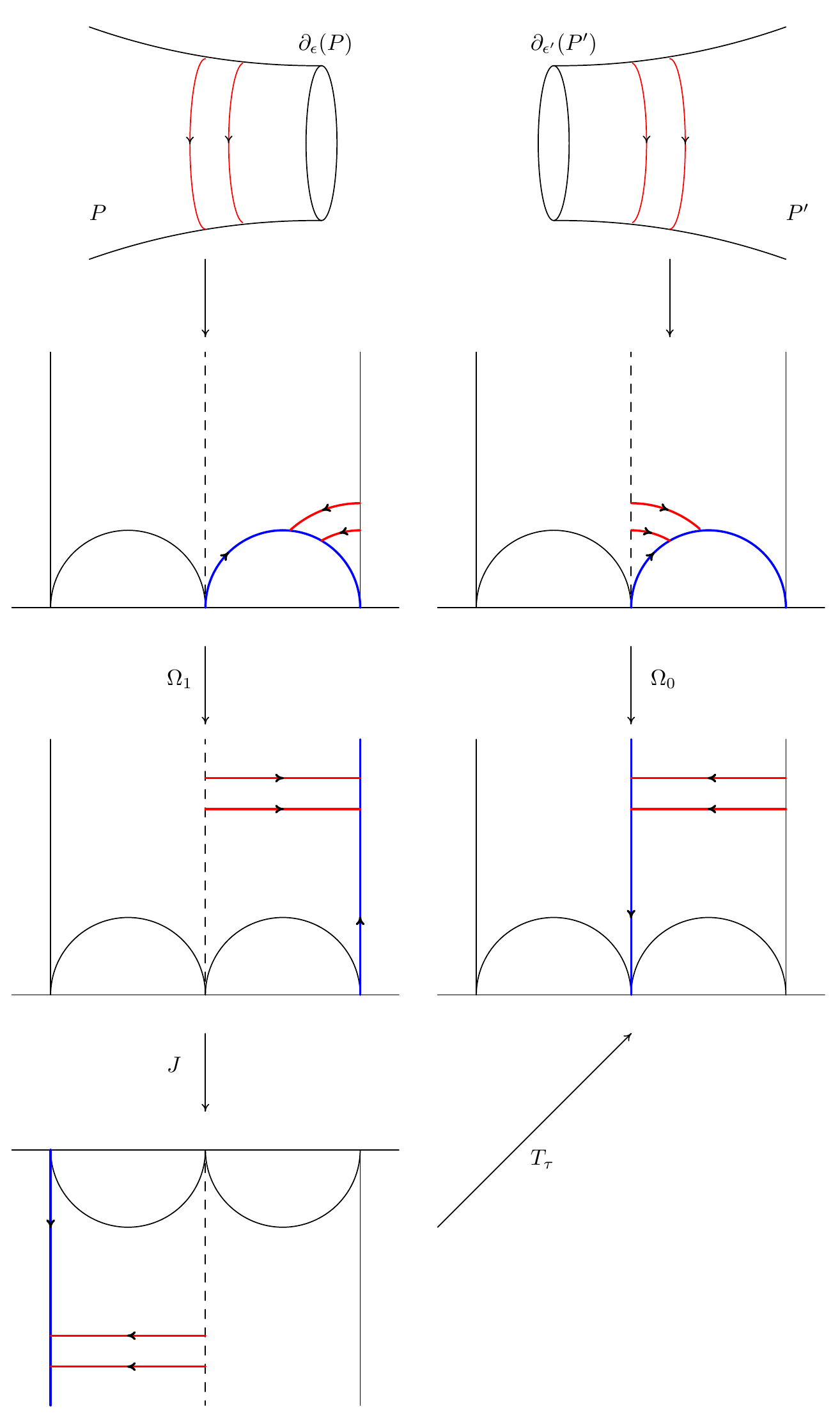} 
\caption{The gluing construction when $\e = 1$ and $\e' = 0$.}
\label{figure4_10} 
\end{figure}

With our usual pants decomposition $\P$, fix homeomorphisms $\Phi_i$ from the interior of each pair of pants $P_i$ to $\mathbb P$. This identification induces a labelling of the three boundary components of $ P_{i}$ as $0, 1, \infty$ in some order, fixed from now on. We denote the boundary labelled $\e \in \{0,1,\infty \}$ by $\dd_{\e} P_{i}$. The identification also induces a colouring of the two right angled hexagons whose union is $P_i$, one being white and one being black. Suppose that the pants $P, P' \in \P$ are adjacent along the pants curve $\s$ meeting along boundaries $\dd_{\e}P$ and $\dd_{\e'}P'$. (If $P =P'$ then clearly $\e \neq \e'$.) The gluing across $\s$ will be described by a complex parameter $\tau$ with $\Im \tau >0$, called the \emph{plumbing parameter} of the gluing. 

Let $\Delta_0 \subset \HH$ be the ideal `white' triangle with vertices $0,1,\infty$. Notice that there is a unique orientation preserving symmetry $\Omega_{\epsilon}$ of $\Delta_{0}$ which sends  the vertex $\epsilon \in \{0,1,\infty\}$ to $ \infty$: 
\begin{equation*}
    \label{eqn:standardsymmetries}
\Omega_{0} = \begin{pmatrix}
	1  &  -1  \\
	1 &  0   
	\end{pmatrix},  \ \ 
 \Omega_{1} =  \begin{pmatrix}
0 &  -1  \\
	1 &  -1 \end{pmatrix},  \ \ 
	\Omega_{\infty} = \Id = \ \begin{pmatrix}
	1  &  0  \\
	0 &  1 	\end{pmatrix}.
\end{equation*}

As described in Figure~\ref{figure4_10}, first we use the maps $\Omega_{\epsilon}$ to reduce to the case $\e = \e' = \infty$. In that case, we first need to reverse the direction in the left triangle $\Delta_{0}$, by the map $J$ which is a rotation about the origin of an angle $\pi$, and then we should translate it, by the map $T_{\tau}$ where 
\begin{equation*}
\label{eqn:standardsymmetries1} J = 
\begin{pmatrix}
	-i & 0 \\
	0 & i 
\end{pmatrix}
, \ \ T_{\tau}= 
\begin{pmatrix}
	1 & \tau \\
	0 & 1.
\end{pmatrix}
\end{equation*}

The gluing map between the pants $P, P' \in \P$ is then described by $$\Omega_{\e}^{-1} J^{-1}T_{\tau}^{-1}\Omega_{\e'}.$$

For a general discussion, we refer to Section 4 and 5 of \cite{Maloni-Series}. The recipe for gluing two pants apparently depends on the direction of travel across their common boundary. Lemma 4.2 in \cite{Maloni-Series} shows that, in fact, the gluing in either direction is implemented by the same recipe and uses the same parameter $\tau$.

\begin{Remark}[Relationship with Kra's construction]
    As explained in detail in Section 4.4 of \cite{Maloni-Series}, Kra's plumbing construction (see Kra \cite{Kra2}) is essentially identical to our construction. The difference is that we implement the gluing in the upper half space $\HH$ without first mapping to the punctured disk $\mathbb D_{\ast}$. In particular the precise relationship between our plumbing parameter $\tau$ and Kra's one $t_{K}$ is given by $$\tau = -\frac{i}{\pi} \log t_{K}.$$ 
\end{Remark}

\subsection{Top Terms' Relationship} \label{ssub:top}

We can now state the main result of our previous work which will be fundamental for the proof of our main theorems. The plumbing construction described in Section \ref{sub2.2} endows $\Sigma$ with a projective structure whose associated holonomy representation $\rho_{\ttau} \co \pi_1(\Sigma) \to PSL(2,\C)$ depends holomorphically on the plumbing parameters ${\ttau} = (\tau_{1},\ldots,\tau_{\xi})$. In particular, the traces of all elements $\rho(\g) , \g \in \pi_1(\Sigma)$, are polynomials in the $\tau_i$. Theorem A of \cite{Maloni-Series} is a very simple relationship between the coefficients of the top terms of $\rho(\g) $, as polynomials in the $\tau_i$, and the Dehn--Thurston coordinates of $\g$ relative to $\P$. 

\begin{Theorem}[Top Terms' Relationship]\label{top}
    Let $\gamma$ be a connected simple closed curve on $\Sigma$,such that its Dehn--Thurston coordinates are $\i(\g) = (q_1,p_1,\ldots,q_{\xi},p_{\xi})$. If $\g$ not parallel to any of the pants curves $\s_i$, then $\tr \rho(\g)$ is a polynomial in $\tau_{1}, \cdots , \tau_{\xi}$ whose top terms are given by: 
    \begin{align*}
    	\tr \rho(\g) &=\pm i^{q} 2^{h} \Bigl(\tau_{1}+\frac{(p_{1}-q_{1})}{q_{1}}\Bigr)^{q_{1}} \cdots \Bigl(\tau_{\xi}+\frac{(p_{\xi}-q_{\xi})}{q_{\xi}}\Bigr)^{q_{\xi}} + R,\\
    	& = \pm i^{q} 2^{h} \left(\tau_{1}^{q_{1}}\cdots\tau_{\xi}^{q_{\xi}} +\sum_{i=1}^{\xi} (p_{i}-q_{i}) \tau_{1}^{q_1} \cdots \tau_i ^{q_{i}-1}\cdots\tau_{\xi}^{q_{\xi}} \right) + R 
    \end{align*}
    where 
    \begin{itemize}
    	\item $q =\sum_{i=1}^{\xi}q_{i}>0$; 
    	\item $R$ represents terms with total degree in $\tau_{1} \cdots \tau_{\xi}$ at most $q - 2$ and of degree at most $q_{i}$ in the variable $\tau_{i}$; 
    	\item $h= h(\g)$ is the total number of $scc$-arcs in the standard representation of $\g$ relative to $\P$, see below. 
    \end{itemize}
    If $q(\g) =0$, then $\gamma = \s_i$ for some $i$, $\rho(\g)$ is parabolic, and $\tr \rho(\g) = \pm 2$.
\end{Theorem}

The non-negative integer $h = h(\g)$ is defined as follows. The curve $\g$ is first arranged to intersect each pants curve minimally. In this position, it intersects a pair of pants $P$ in a number of arcs joining the boundary curves of $P$. We call one of these an \emph{$scc$-arc} (short for same-(boundary)-component-connector) if it joins one boundary component to itself, and denote by $h$ the total number of $scc$-arcs, taken over all pants in $\P$. Note that some authors call the scc-arcs \textit{waves}.

\begin{Remark}
    As noted in Section 4.2 \cite{Maloni-Series}, with our convention the base point for the gluing construction is when $\Re \t_i = 1$. It would be more natural to have, as base point, $\Re \t_i = 0$. That can be achieved by changing the fundamental domain for the standard triply punctured sphere. In particular, one should have as the white triangle $\Delta_0$ the set $\{z\in \C | \Re z \in (-\frac{1}{2}, \frac{1}{2}), |z| > \frac{1}{2}\}$. This new parameter, equal the old one minus 1, would also make the formula above neater. In fact the formula, with this new parameter, also called call $\t_i$, becomes:
    \begin{align*}
    	\tr \rho(\g) &=\pm i^{q} 2^{h} \Bigl(\tau_{1}+\frac{p_{1}}{q_{1}}\Bigr)^{q_{1}} \cdots \Bigl(\tau_{\xi}+\frac{p_{\xi}}{q_{\xi}}\Bigr)^{q_{\xi}} + R,\\
    	& = \pm i^{q} 2^{h} \left(\tau_{1}^{q_{1}}\cdots\tau_{\xi}^{q_{\xi}} +\sum_{i=1}^{\xi} p_{i}\tau_{1}^{q_1} \cdots \tau_i ^{q_{i}-1}\cdots\tau_{\xi}^{q_{\xi}} \right) + R 
    \end{align*}
    
    From now on we will use this new parameter which is equal the $\t_i--$parameter in \cite{Maloni-Series} minus 1.
\end{Remark}

\subsection{Maskit embedding}\label{sub:maskit}

In this section we recall the definition of the \emph{Maskit embedding} of $\Sigma$, following Series' article~\cite{Series1}, see also~\cite{Maskit}. Let $\RR(\Sigma)$ be the set of representations $\rho \co \pi_1(\Sigma) \to PSL(2,\C)$ modulo conjugation in $PSL(2,\C)$. Let $\M \subset \cal R$ be  the subset of representations for which: 
\begin{enumerate}\renewcommand{\labelenumi}{(\roman{enumi})}
\item the group $G= \rho\left(\pi_1(\Sigma)\right)$ is discrete (Kleinian) and $\rho$ is an isomorphism, 
\item the images of $\s_i$, $i=1, \ldots, \xi$, are parabolic,
\item all components of the regular set $\Omega(G)$ are simply connected and there is exactly one invariant component  $\Omega^+(G)$,
\item the quotient $\Omega(G)/G$ has $k + 1$ components (where $k = 2g-2+n$ if $\Sigma = \Sigma_{(g,n)}$), $\Omega^+(G)/G$ is homeomorphic to $\Sigma$ and the other components are triply punctured spheres.
\end{enumerate}
In this situation, see for example Section 3.8 of Marden~\cite{Marden}, the corresponding $3$--manifold $M_{\rho} = \HH^3/G$ is topologically $\Sigma \times (0,1)$. Moreover $G$ is a geometrically finite cusp group on the boundary (in the algebraic topology) of the set of quasifuchsian representations of $\pi_1(\Sigma)$. The `top' component $\Omega^+/G $ of the conformal boundary may be identified to $\Sigma \times \{1\}$ and is homeomorphic to $\Sigma$. On the `bottom' component  $\Omega^-/G $, identified to  $\Sigma \times \{0\}$, the pants curves $\s_1,\ldots, \s_\xi$ have been pinched, making  $\Omega^-/ G$ a union of $k$ triply punctured spheres glued across punctures corresponding to the curves $\s_i$. The conformal structure on $\Omega^+/G $, together with the pinched curves $\s_1,\ldots, \s_\xi$, are the \emph{end invariants} of $M_{\rho}$ in the sense of Minsky's ending lamination theorem. Since a triply punctured sphere is rigid, the conformal structure on $\Omega^-/ G$ is fixed and independent of $\rho$, while the structure on  $\Omega^+/ G$  varies. It follows from standard Ahlfors--Bers theory, using the Measurable Riemann Mapping Theorem (see again Section 3.8 of~\cite{Marden}), that there is a unique group corresponding to each possible conformal structure on $\Omega^+/ G$. Formally, the \emph{Maskit embedding} of the \Teich space of $\Sigma$ is the map $\teich(\Sigma) \to \cal R$ which sends a point  $X \in \teich(\Sigma)$ to the unique group $G \in \M$ for which $\Omega^+/ G$ has the marked conformal structure $X$. 

\subsubsection{Relationship between the plumbing construction and the Maskit embedding}

In the Section \ref{sub2.2}, given a pants decomposition $\PC = \{ \s_1, \ldots,\s_{\xi}\}$ of $\Sigma$, we constructed a family of projective structures on $\Sigma$, to each of which is associated a natural holonomy representation $\rho_{\ttau} \co \pi_1(\Sigma) \to PSL(2,\C)$. Proposition 4.4 of~\cite{Maloni-Series} proves that our plumbing construction described above, for suitable values of the parameters, gives exactly the \emph{Maskit embedding} of $\Sigma$.  

\begin{Proposition}[Proposition 4.4 \cite{Maloni-Series}]
    Suppose that ${\ttau} \in \HH^{\xi}$ is such that the associated developing map $Dev_{\ttau} \co \tilde \Sigma \to \hat \C$ is an embedding. Then the holonomy representation $ \rho_{\ttau}$ is a group isomorphism and $G = \rho_{\ttau}(\pi_1(\Sigma)) \in \cal M$.
\end{Proposition}

\subsection{Three manifolds and pleating rays} \label{ssub:manifolds}

Let $M$ be a hyperbolic 3--manifold, that is a complete 3-dimensional Riemannian manifold of constant curvature $-1$ such that the fundamental group $\pi_{1}(M)$ is finitely generated. We exclude the somewhat degenerate case $\pi_{1}(M)$ has an abelian subgroup of finite index, that is $\pi_{1}(M)$ is an elementary Kleinian group. An important subset of $M$ is its \textit{convex core} $\CC_{M}$ which is the smallest, non-empty, closed, convex subset of $M$. The boundary $\partial\CC_{M}$ of this convex core is a surface of finite topological type whose geometry was described by W. Thurston \cite{Thurston}. Note that given a hyperbolic $3$--manifold $M =\HH^{3}/G$, we can also define the convex core as the quotient $\CH(\Lambda)/G$ where $\CH(\Lambda)$ is the convex hull of the limit set $\Lambda = \Lambda(G)$ of $G$, see \cite{Epstein} for a detailed discussion on the pleated structure of the boundary of the convex core. If $M$ is geometrically finite, then there is a natural homeomorphism between each component of $\partial\CC_{M}$ and each component of the conformal boundary $\Omega/G$ of $M$. Each component $F$ of $\partial\CC_{M}$ inherits an induced hyperbolic structure from $M$. Thurston also proved such each component is a pleated surface, that is a hyperbolic surfaces which is totally geodesic almost everywhere and such that the locus of points where it fails to be totally geodesic is a geodesic lamination. Formally a pleated surface is defined in the following way.

\begin{Definition}
    A \textit{pleated surface} with topological type $S$ in a hyperbolic 3--dimensional manifold $M$ is a map $f \co S \to M$ such that:
    \begin{itemize}
        \item the path metric obtained by pulling back the hyperbolic metric of $M$ by $f$ is a hyperbolic metric $m$ on $S$;
        \item there is an $m$-geodesic lamination $\l$ such that $f$ sends each leaf of $\l$ to a geodesic of $M$ and is totally geodesic on $S - \l$.
    \end{itemize}
    In this case, we say that the pleated surface $f$ admits the geodesic lamination $\l$ as a \textit{pleated locus} and $\l$ is called the \textit{bending lamination} and the images of the complementary components of $\l$ are called the \emph{flat pieces} (of the pleated surface). 
\end{Definition}

The bending lamination of each component of $\partial\CC_{M}$ carries a natural transverse measure, called the \textit{bending measure} (or \textit{pleating measure}). In the case $M = \Sigma \times \R$, there are two components $\partial^+\CC_{M}$ and $\partial^-\CC_{M}$ of $\partial\CC_{M}$ and we will denote $\pl^\pm \in \ML(\Sigma)$ the respective pleating measure on each one of them.

We will deal with manifolds for which the bending lamination is \emph{rational}, that is, supported on closed curves. The subset of rational measured laminations is denoted $\ML_{\Q}(\Sigma) \subset \ML(\Sigma)$ and consists of measured laminations of the form $\sum_{i = 1}^{k} a_i \delta_{\g_i}$, where the curves $\g_i \in \S(\Sigma)$ are disjoint and non-homotopic, $a_i \geq 0$, and $ \delta_{\g_i}$ denotes the transverse measure which gives weight $1$ to each intersection with $\g_i$. If $\sum_{i = 1}^{k} a_i \delta_{\g_i}$ is the bending measure of a pleated surface $\Sigma$, then $a_i$ is the angle between the flat pieces adjacent to $\g_i$, also denoted $\theta_{\g_i}$. In particular, $\theta_{\g_i} = 0$ if and only if the flat pieces adjacent to $\g_i$ are in a common totally geodesic subset of $\dd \CC/G$. We take the term pleated surface to include the case in which a closed leaf $\g$ of the bending lamination maps to the fixed point of a rank one parabolic cusp of $M$. In this case, the image pleated surface is cut along $\g$ and thus may be disconnected. Moreover the bending angle between the flat pieces adjacent to $\g$ is $\pi$. See discussion in \cite{Series1} or \cite{Choi-Series}.

An important result, due to Bonahon and Otal, about the existence of hyperbolic manifolds with prescribed bending laminations is the following. Recall that a set of curves $\{\g_1, \ldots, \g_n\}$ in a surface $\Sigma$ \textit{fills} the surface if for any $\gamma \in \S(\Sigma)$ there exist $j \in \{1, \ldots,n\}$ such that $i(\g, \g_j) \neq 0$.

\begin{Theorem}[Theorem 1 of \cite{Bon-Otal}] \label{thm:bo} 
    Suppose that $M $ is $3$--manifold homeomorphic to $\Sigma \times (0,1)$, and that $\xi^{\pm} = \sum_i a^{\pm}_i \g^{\pm}_i \in \ML_{\Q}(\Sigma) $. Then there exists a geometrically finite group $G$ such that $M = \HH^3/G$ and such that the bending measures on the two components $\dd \CC^{\pm}(G)$ of $\dd \CC(G)$ equal $\xi^{\pm}$ respectively, if and only if $a^{\pm}_i \in (0,\pi]$ for all $i$ and $\{ \g^{\pm}_i, i=1,\ldots,n\}$ fill up $\Sigma$ (i.e. if $i(\xi^+, \gamma) + i(\xi^-,\gamma)>0$ for every $\gamma \in \S$). If such a structure exists, it is unique. 
\end{Theorem}

Specialising now to the Maskit embedding $\M = \M(\Sigma)$, let $\rho = \rho_{\ttau}$ where $\ttau = (\t_1,\ldots,\t_\xi) \in \C^{\xi}$ be a representation $\rho\co\pi_1(\Sigma) \to SL(2,\C)$ such that the image $G = G(\t_1,\ldots,\t_\xi) \in \M$. The boundary $\dd\cal C (G)$ of the convex core has $\xi + 1$ components, one $\partial^+ \cal C$ facing $\Omega^+/ G$ and homeomorphic to $\Sigma$, and $\xi$ triply punctured spheres whose union we denote $\partial^- \cal C$. The induced hyperbolic structures on the components of $\partial^- \cal C$ are rigid, while the structure on $\partial^+ \cal C$ varies. We recall that we denoted $\pl^+(G) \in \ML(\Sigma)$ the bending lamination of $\partial^+ \cal C$. Following the discussion above, we view $\partial^- \cal C$ as a single pleated surface with bending lamination $\pi(\s_1+\ldots+\s_\xi)$, indicating that the triply punctured spheres are glued across the annuli whose core curves $\s_1, \ldots, \xi\s_2$ correspond to the parabolics $S_i \in G$. 

\begin{Corollary}\label{cor:exist} 
A lamination $\eta \in \ML_{\Q}(\Sigma)$ is the bending measure of a group $G \in \M$ if and only if $i(\eta, \s_1) ,\ldots, i(\xi, \s_{\xi})>0$. If such a structure exists, it is unique. 
\end{Corollary}

We call $\eta \in \ML_{\Q}(\Sigma)$ \emph{admissible} if $i(\eta, \s_1) ,\ldots, i(\xi, \s_{\xi})>0$.

\subsubsection{Pleating rays} \label{ssub:pleating}

Denote the set of projective measured laminations on $\Sigma$ by $\PML = \PML(\Sigma) $ and the projective class of $\eta = a_1 \g_1 + \ldots + a_{k} \g_k \in \ML$ by $[\eta]$. The \emph{pleating ray} $\P = \P_{[\eta]}$ of $\eta \in \ML$ is the set of groups $G \in \M$ for which $ \pl^+(G) \in [\eta]$. To simplify notation we write $ \P_{\eta}$ for $ \P_{[\eta]}$ and note that $ \P_{\eta}$ depends only on the projective class of $\eta$, also that $ \P_{\eta}$ is non-empty if and only if $\eta$ is admissible. In particular, we write $\P_{\g}$ for the ray $\P_{[\delta_{\g}]}$. As $ \pl^+(G)$ increases, $ \P_{\eta}$ limits on the unique geometrically finite group $G_{\rm cusp}(\eta)$ in the algebraic closure $\overline{ \M}$ of $\M$ at which at least one of the support curves to $\eta$ is parabolic, equivalently so that $ \pl^+(G) = \th(a_1 \g_1 + \ldots + a_{k} \g_k)$ with $\max \{ \th a_1, \ldots, \th\a_k\} = \pi$. We write $\overline{ \P_{\eta} }= \P_{\eta} \cup G_{\rm cusp}(\eta)$. 

The following key lemma is proved in Proposition 4.1 of Choi and Series \cite{Choi-Series}, see also Lemma 4.6 of Keen and Series \cite{Keen-Series1}. The essence is that, because the two flat pieces of $\dd \CC(G)$ on either side of a bending line are invariant under translation along the bending line, the translation can have no rotational part. 

\begin{Lemma}\label{lemma:realtrace} 
    If the axis of $g \in G$ is a bending line of $\dd \CC(G)$, then $\tr(g) \in \R$. 
\end{Lemma}

Notice that the lemma applies even when the bending angle $\th_{\g}$ along $\g$ vanishes. Thus if $G \in {{\overline \P}_{\eta_{\g_1,\ldots,\g_k}}}$, where $\eta_{\g_1,\ldots,\g_k} = \sum_{i =1}^{k} a_i \delta_{\g_i}$,  we have $\tr g \in \R$ for any $g \in G$ whose axis projects to a curve $\g_i$, $i=1,\ldots, k$.

In order to compute pleating rays, we need the following result which is a special case of Theorems B and C of~\cite{Choi-Series}, see also~\cite{Keen-Series1}. Recall that a codimension-$p$ submanifold $N \hookrightarrow \C^n$ is called \emph{totally real} if it is defined locally by equations $\Im f_i = 0, i =1,\ldots,p$, where $f_i, i = 1, \ldots,n$ are local holomorphic coordinates for $\C^n$. As usual, if $\g$ is a bending line we denote its bending angle by $\theta_{\g}$. Recall that the \emph{complex length} $\l(A)$ of a loxodromic element $A \in SL(2,\C)$ is defined by $ \tr A = 2 \cosh \frac{\l(A)}{2}$, see e.g.~\cite{Choi-Series} for details. By construction, $\P_{\gamma_1,\ldots, \g_k} \subset \M \subset \RR(\Sigma)$.

\begin{Theorem}	\label{thm:cscoords} 
    The complex lengths $\l(\gamma_1),\ldots, \l(\g_k)$ are local holomorphic coordinates for $\RR(\Sigma)$ in a neighbourhood of $\P_{\eta_{\g_1,\ldots,\g_k}}$. Moreover $\P_{\eta_{\g_1,\ldots,\g_k}} $ is connected and is locally defined as the totally real submanifold $\Im \tr \gamma_i = 0, i=1,2$ of $\R$. Any $k$--tuple $(f_1,f_2, \ldots, f_k)$, where $f_i$ is either the hyperbolic length $\Re \l(\g_i) $ or the bending angle $ \theta_{\g_i}$, are global coordinates on $\P_{\eta_{\g_1,\ldots,\g_k}}$.
\end{Theorem}

This result extends to $\overline{\P}_{\eta_{\g_1,\ldots,\g_k}}$, except that one has to replace $ \Re \l(\g_i) $ by $\tr \gamma_i$ in a neighbourhood of a point for which $\gamma_i$ is parabolic. In fact, as discussed in~\cite[Section 3.1]{Choi-Series}, complex length and traces are interchangeable except at cusps (where traces must be used) and points where a bending angle vanishes (where complex length must be used). The parameterisation by lengths or angles extends to $\overline {\P}_{\gamma_1,  \ldots, \g_k}$.

Notice that the above theorem gives a local characterisation of $\overline {\P}_{\eta_{\g_1,\ldots,\g_k}}$ as a subset of the representation variety $\RR$ and not just of $\M$. In other words, to locate $\P$, one does not need to check whether nearby points lie \emph{a priori} in $\M$; it is enough to check that the traces remain real and away from $2$ and that the bending angle on one or other of $\theta_{\g_i}$ does not vanish. As we shall see, this last condition can easily be checked by requiring that further traces be real valued.

\section{Main theorems} \label{sec3} 

In this section we will prove our main results. As explained in the Introduction, we will extend to a general hyperbolic surface $\Sigma_{g, n}$ the results proved by Series \cite{Series1} for the case of a twice punctured torus $\Sigma_{1,2}$. As already observed by Series, almost all the results of Section 6 \cite{Series1} generalise straightforwardy, but for Section 7 \cite{Series1} some non-trivial extensions are needed. So we will only restate the most important theorems of Section 6 without proof and refer to the original paper for a more detailed discussion. Almost all the results of Section 7 still remain true, but we will discuss how to generalise them more deeply. In addition, we find how to include the case of `exceptional curves' in the proof of the main theorems (so we will not need to discuss that case separately). We will also correct some misprints in \cite{Series1}. All these remarks will be explained in detail later on.

The key idea for proving these theorems is to understand the geometry of the top component $\partial^+ \CC(G)$ of the convex core for groups $G = G_{\eta}(\th) \in \P_{\eta} \subset \M$ as $\th \to 0.$ Recall that the definition of $\M$ depends on the choice of a pants decomposition $\PC = \{ \s_1, \ldots,\s_{\xi}\}$, which tells us the curves which will be pinched in the bottom surface of the associated manifold. Before stating the results, we need to fix some notation. We will use Series' notation, so that the interested reader can refer to the paper \cite{Series1} more easily. 

\begin{Notation}
    Given a quantity $X = X(\s_i)$ which depend on the pants curve $\s_i\in \PC$, we will write $\mathbf{X(\s_i) = O(\theta^{e})},$ meaning that $X \leqslant c\theta^{e}$ as $\theta \to 0$ for some constant $c > 0$, where $e$ is an exponent (usually $e = 0, 1$).
\end{Notation}

\begin{Remark}
    Note that the estimates below all depends on the lamination $\eta$. So, more precisely, one has $X \leqslant c(\eta)\theta^{e}$. However it is easily seen, by following through the arguments, that the dependence on $\eta$ is always of the form $X(\s_i) \leqslant cq^{e}\theta^{e}$, where $q = i(\s_i,\eta)$ and where, now, $c$ is a universal constant independent of $\eta$. The dependence of the constants on $\eta$ is not important for our argument, but it may be useful elsewhere. 
\end{Remark}

The main theorem in Section 6 of \cite{Series1} is Proposition 6.1. The proof of this result relies on three other main lemmas proved in the same section, namely Proposition 6.6 and 6.11 for the asymptotic behaviour of the imaginary part of the parameters $\t_i$ and Proposition 6.14 for the real part. (See Series' article for the proofs.) The only remark is that the role played in $\Sigma_{1,2}$ by the curve $\g_{T}$ for the pants curves $\s_1$ and $\s_2$ should be replaced by the curves $D_i$, dual to the pants curve $\s_i$. In fact, the important property of $\g_{T}$ is that it intersects $\s_i$ minimally. In particular, for the second part of the proof of Corollary 6.5 instead of using $\Tr [T, S_i^{-1}]$ you should use $\Tr D_i$, and for Proposition 6.18 instead of calculating $i_{\s_i}(\g,{T})$ you should deal with $i_{\s_i}(\g, D_i)$. The ideas for the proofs remain however the same. Finally, we remind the reader that the twist parameters $p_i$ used in this article are twice the value of the `old' parameters (again called $p_i$) used by Series in \cite{Series1}. The parameters $p_i$ we are using in this article are the twist parameters using D. Thurston's standard position (as defined in \cite{Maloni-Series}). A generalisation of Proposition 6.1 of \cite{Series1} is the following. 

\begin{Theorem}\label{Prop6_1} 
    Let $\eta = \sum_{i =1}^{k} a_i \delta_{\g_i}$ be an admissible rational measured lamination on the surface $\Sigma = \Sigma_{g,b}$ and let $G =G_{\eta}(\th)$ be the unique group in $\M$ with $\pl^+(G) = \th \eta$. Then, as $\th \to 0$, we have:
    $$\Re \tau_i = - \frac{p_i(\eta)}{q_i(\eta)} + O(1) \;\;\;\;\; {\rm and} \;\;\;\;\;\; \Im \tau_i = \frac{4 + O(\th)}{\th q_i(\eta)},$$ where $O(1)$ denotes a universal bound independent of $\eta$.
\end{Theorem}

\begin{Corollary}\label{Corol6_1}
    With the same hypothesis as Theorem \ref{Prop6_1}, as $\th \to 0$, we have:
    $$\frac{\Im \tau_i\Re \tau_i}{\Re \tau_i\Im \tau_i} = \frac{p_i}{p_j} + O(\th) \;\;\;\;\; {\rm and} \;\;\;\;\;\; \frac{\Im \tau_i}{\Im \tau_j} = \frac{q_j}{q_i} + O(\th).$$
\end{Corollary}

This result is enough in order to prove Theorem \ref{thmC}. We will follow Series' proof very closely.

\begin{proof}[Proof of Theorem \ref{thmC}]
 Let $\eta = \sum_{1}^{\xi} a_i \delta_{\g_i}$ be admissible and let $G = G_{\eta}(\th)$ be the unique group for which $\pl^+(G) = \th \eta$. Let $h(\th)$ denote the hyperbolic structure of $\dd \CC^+(G)$. Let $l^+_{\s_i}$ be the hyperbolic length of the geodesic representative of $\s_i$ on the hyperbolic surface $\partial^+ \CC(G)$. Since $l^+_{\s_i} \to 0$, for all $i=1, \ldots, \xi$, the limit of the structures $h (\th)$ in $\PML(\Sigma)$ is in the linear span of $\d_{\s_1},\ldots, \d_{\s_\xi}$. We want to prove that the limit is the barycentre $\sum_{1}^{\xi} \delta_{\s_i}$.

    Let $\d, \d' \in \S$. Since $\s_1,\ldots,\s_\xi$ are a maximal set of simple curves on $\Sigma$, the thin part of $h(\th)$ is eventually contained in collars $A_i$ around $\s_i$ of approximate width $\log(\frac{1}{l^+_{\s_i}})$ and the lengths of $\d, \d'$ outside the collars $A_i$ are bounded (with a bound depending only on the combinatorics of $\d,\d'$ and hence the canonical coordinates $\i(\d), \i(\d')$). By the results of Section 6.4 of \cite{Series1}, the twisting around $A_i$ is bounded. We deduce that for any curve transverse to $\s_i$ we have 
    \begin{equation}\label{eqn:collar} 
        l^+_{\d} = 2 \sum_{i=1}^{\xi} q_i(\d) \log(\frac{1}{l^+_{\s_i}}) + O(1), 
    \end{equation}
see for example Proposition 4.2 of Diaz and Series~\cite{Diaz-Series}. By Theorem~\ref{Prop6_1} we have $\frac{l^+_{\s_i}}{l^+_{\s_j}}\to \frac{q_j(\eta)}{q_i(\eta)}$, and since $\eta$ is admissible, $q_i(\xi)>0$ for $i=1, \ldots, \xi$. Thus $\frac{ \log l^+_{\s_i}}{\log l^+_{\s_j}}\to 1$. Hence $$\frac{l^+_{\d}}{l^+_{\d'}} \to \frac{\sum_{i=1}^{\xi} q_i(\d)}{\sum_{i=1}^{\xi} q_i(\d')} = \frac{i(\d, \sum_{1}^{\xi} \delta_{\s_i})}{i(\d', \sum_{1}^{\xi} \delta_{\s_i})}.$$ The result follows from the definition of convergence to a point in $\PML(\Sigma)$.  
\end{proof}

The next results are the key tools for the proofs of Theorems \ref{thmA}. We need to fix more notation. Suppose that $\gamma$ is a bending line of $\dd \CC^{+}(G)$ for a group $G(\ttau) \in \P_{\eta}$. The Top Terms' Relationship \ref{top}, together with the condition $\tr \g \in \R$ of Lemma \ref{lemma:realtrace}, gives asymptotic conditions for $\ttau \in \P_{\xi}$, in terms of the canonical coordinates ${\bf i}(\g)$ of $\g$. In particular, for $\ttau = (\tau_{1}, \ldots, \tau_{\xi}) \in \C^{\xi}$ set $\tau_i = x_i+i y_i,\;\; \rho = \| (y_{1}, \ldots, y_{\xi})\| = (y_{1}^2+ \ldots+ y_{\xi}^2)^{\frac{1}{2}}$, and $\eta_i= \frac{y_i}{\rho}$. Define 
\begin{align*}
	E_{\gamma} (\tau_{1}, \ldots, \tau_{\xi}) = \eta_2\cdots\eta_{\xi}(q_1x_1 + p_1) + \ldots + \eta_{1}\cdots\eta_{\xi-1} (q_{\xi} x_{\xi} + p_{\xi})\\
	= \eta_1\cdots\eta_{\xi}\sum_{i = 1}^{\xi}\frac{(q_ix_i + p_i)}{\eta_{i}}, 
\end{align*}
where as usual ${\bf i}(\g) = (q_1(\g), p_1(\g),\ldots, q_{\xi}(\g), p_{\xi}(\g))$ and $y_i> 0, i=1,\ldots,\xi$. 

The reason why we introduced this notation is the following result, which generalises Proposition 7.1 of \cite{Series1}. Again Series' proof extends clearly to our case.
 
\begin{Proposition}\label{prop7_1} 
    Suppose that $\eta \in \MLQ$ is an admissible lamination, that $G(\tau_{1}, \ldots, \tau_{\xi}) \in \P_{\eta}$ has bending measure $\pl^+(G) = \th \eta$, and that $\g $ is a bending line of $\eta$. Then, as $\th \to 0$, we have 
    $$E_{\gamma} (\tau_{1}, \ldots, \tau_{\xi}) = O(\th).$$ 
\end{Proposition}

Now we want to locate the pleating ray $\P_{\eta}$ where $\eta = \sum_{i = 1}^{k} a_i \g_i$. If $G \in \P_{\g_{1}, \ldots,\g_{k}}$, then $ \dd \CC^+(G) - \{\g_{1}, \ldots,\g_{k}\}$ is flat, so that not only $\g_{1}, \ldots,\g_{k}$, but also any curve $\d \in \mathrm{lk}(\g_{1}, \ldots,\g_{k})$, is a bending line for $G$, where $\mathrm{lk}(\g_{1}, \ldots,\g_{k})$ denotes the link of the simplex $(\g_{1}, \ldots,\g_{k})$ in the complex of curves $\CC(\Sigma)$. One can think of it as the set of all curves $\d \in \S = \S(\Sigma)$ disjoint from $\g_{1}, \ldots,\g_{k}$. Thus $\ttau = (\tau_{1}, \ldots, \tau_{\xi})$ is constrained by the equations $$\Im \tr \g_{i} =\Im \tr \d =0 \;\;\;\; \forall i = 1,\ldots,k, \;\forall \d \in \mathrm{lk}(\g_{1}, \ldots,\g_{k})$$ and hence, using the Proposition \ref{prop7_1}, it is constrained by the following equations 
\begin{equation*}
	E_{\gamma_{i}} (\tau_{1}, \ldots, \tau_{\xi}) + O({\th}) = 0, \ \ {\rm and} \ \ E_{\d} (\tau_{1}, \ldots, \tau_{\xi}) + O({\th})= 0 
\end{equation*}
for all $\d \in \mathrm{lk}(\g_{1}, \ldots,\g_{k})$ and for $i = 1,\ldots,k$. Now we would like to describe how to solve these equations simultaneously for $\tau_{1}, \ldots, \tau_{\xi}$.

Following the analysis in Section 7 of \cite{Series1}, we recall that for any curve $\omega \in \S$ we have $$E_{\omega}(\tau_{1}, \ldots, \tau_{\xi})= \i(\omega) \cdot {\bf u},$$ where $\i(\omega)= (q_1(\omega), p_1(\omega),\ldots, q_{\xi}(\omega), p_{\xi}(\omega))$ and 
\begin{align*}
    {\bf u}\; = \;(u_{1 1}, u_{1 2}, \ldots, u_{\xi 1}, u_{\xi 2})\;= \;\eta_1\cdots\eta_{\xi} (\frac{x_1}{\eta_1},\; \frac{1}{\eta_1},\;\ldots,\; \frac{x_{\xi}}{\eta_{\xi}},\; \frac{1}{\eta_{\xi}}) \\	
    = (\eta_2\cdots\eta_{\xi} x_1,\;\; \eta_2\cdots\eta_{\xi},\;\ldots,\;\; \eta_1\cdots\eta_{\xi-1} x_{\xi},\;\;\eta_1\cdots\eta_{\xi-1}) 
\end{align*}
with $x_i = \Re \t_i,\; \eta_i = \frac{\Im \t_i}{\rho}$ as above. We will use linear algebra and Thurston's symplectic form $\Omega_{\Th}$ to solve the equations 
$$\i(\g_i) \cdot {\bf u} = 0,\;\;\;\;\i(\d) \cdot {\bf u} = 0$$
for all $\d \in \mathrm{lk}(\g_{1}, \ldots,\g_{k})$ and for $i = 1,\ldots,k$. As already noted in Section \ref{ssub:sympl}, this symplectic form induces a map $\R^{2\xi} \to \R^{2\xi}$ defined by ${\bf x} = (x_{1},y_{1},\ldots,x_{\xi},y_{\xi}) \to  {\bf x}^* = (y_{1},-x_{1},\ldots,y_{\xi},-x_{\xi})$ such that $$\Omega_{\Th} (\i(\g),\i(\delta) ) = \i(\g) \cdot \i(\delta)^* $$ where $\cdot$ is the usual inner product on $\R^{2\xi}.$ 

We need the following Lemma, which generalise Lemma 7.2 of \cite{Series1}. See Section 2.6 of Penner \cite{Penner} for a definition of \textit{standard} train tracks. Note that, although not necessary, we will use the language of the curve and marking complexes, since many readers may find it useful. See Section \ref{ssub:complex} for the basic definitions.

\begin{Lemma}\label{pants}\indent\par
    	\begin{itemize}
		\item[(i)] Suppose that ${\bf g} = (\gamma_1,\ldots, \g_k)$ is a simplex in the complex of curves $\CC(\Sigma)$. Then $\gamma_i$ are supported on a common standard train track and $\i(\g_i)$ are independent vectors in $\i(\ML_{\Q}(\Sigma)) \subset (\Z_{+} \times \Z)^{\xi}$.
		\item[(ii)] Given any simplex ${\bf g} = (\gamma_1,\ldots, \g_k)$ in the complex of curves $\CC(\Sigma)$, we can find curves $\g_{k+1},\ldots,\g_{\xi}$, $D_{k+1},\ldots, D_{\xi} \in \mathrm{lk}_{\CC(\Sigma)}({\bf g})$ such that the elements $(\g_{1}, \ldots, \g_{\xi})$ and $(\g_{1}, \ldots, \g_{j}, D_{j+1}, \ldots, D_{\xi})$ with $j = k, \ldots, \xi-1$, are simplices in $\CC(\Sigma)$ and such that the vectors $\i(\g_1), \ldots, \i(\g_\xi)$, $\i(D_{k+1}) \ldots, \i(D_\xi)$ span a subspace of real dimension $2\xi-k$ in $\i(\ML_{\Q}(\Sigma)) \subset (\Z_{+} \times \Z)^{\xi}$.
	\end{itemize}
\end{Lemma}

\begin{proof}
    \noindent $(i)$: Following Series' proof, the disjointness of the curves $\gamma_1,\ldots, \g_k$ tells us they are supported on a common standard train track. The second part of $(ii)$ is proved, as a particular case, in the proof of $(ii)$. 
    
    \noindent $(ii)$: The idea is to complete ${\bf g}$ to a pants decomposition of $\Sigma$ and to consider the dual curves of the pants curves added. In detail let $\gamma_{k+1},\ldots, \g_{\xi}$ be such that $\{\gamma_1,\ldots, \g_{\xi}\}$ is a pants decomposition of $\Sigma$ and let $D_{i}$ be the dual curve of $\g_i$. (Note that $D_{i}$ is disjoint from any pants curve $\g_{j}$ when $j \neq i$ and intersects $\g_{i}$ twice.) Using the language of Masur and Minsky \cite{MM2}, we can say we have chosen a complete, clean marking $\mu = (\g_1,\ldots, \g_\xi;D_1,\ldots, D_\xi)$ (that is a vertex in the marking complex where $\g_1,\ldots, \g_k$ are curves in the base of $\mu$) and we define a path $\mu = \mu_{0}, \mu_{1},\ldots, \mu_{2\xi-k}$ by the requirement $\mu_{i}$ is obtained from $\mu_{i-1}$ by flipping $\g_{k+i}$ and $D_k+i$ for $i = 1, \ldots, \xi-k$. The simplices in the statement of the theorem are then the bases of the markings $\mu_i$ for $ i = 0, \ldots, 2\xi-k$.
    
    We want to show that the vectors $\i(\g_1), \ldots, \i(\g_\xi)$, $\i(D_{k+1}) \ldots, \i(D_\xi)$ are linear independent. Without loss of generality, we can assume the map $\i$ is defined with respect to the marking $\mu_0$. Indeed, if that it is not the case, the change of coordinates between the map $\i$ and a new map $\i'$ defined with respect to a new marking $\mu'$ is a linear map, which doesn't change our conclusion about the linear independence of the vectors. Now the vector $\i(\g_i) = (q_1,p_1, \ldots,q_\xi,p_\xi)$ is defined by $p_i = 1$ and $q_j = p_j = q_i = 0$ for all $j \neq i$ and the vector $\i(D_i) = (q_1,p_1, \ldots,q_\xi,p_\xi)$ is defined by $q_i = 2$ and $q_j = p_j = p_i = 0$ for all $j \neq i$. (See Remark \ref{dual} for a description of the convention on dual curves that we are using.) This proves that the vectors $\i(\g_1), \ldots, \i(\g_\xi),\i(\d_{k+1}) \ldots, \i(\d_\xi)$ are linearly independent.
\end{proof}

Now we can state precisely Theorem \ref{thmA} of the Introduction.

\begin{Theorem}[Theorem A]\label{TA}
    Suppose that $\eta = \sum_{i = 1}^{k} a_i \g_i$ is admissible (and $k \leqslant \xi$). Let $\i(\eta) = (q_1(\eta),p_1(\eta),\ldots,q_\xi(\eta),p_\xi(\eta))$. Let $L_{\eta}: [0,\infty) \to \C^\xi$ be the line $t \mapsto (w_1(t),\ldots, w_\xi(t)) $ where $$w_i(t) = -\frac{p_i}{q_i} + it \frac{q_1}{q_i}.$$ Let $(\tau_1(\th),\ldots, \tau_\xi(\th)) \in \C^\xi$ be the point corresponding to the group $G_{\eta}(\th)$ with $\pl^+(G) = \th \eta$, so that the pleating ray $\cal P_{\eta} $ is the image of the map $p_{\eta}: \th \to (\tau_1(\th),\ldots, \tau_\xi(\th))$ for a suitable range of $\th>0$. Then $\cal P_{\eta}$ approaches $L_{\eta}$ as $\th \to 0$ in the sense that if $t(\th) = \frac{4}{\th q_1}$, then $$ | \Re \tau_i(\th) - \Re w_i(t(\th)) | = O(\th) \ {\rm and}\ | \Im \tau_i(\th) - \Im w_i(t(\th)) | = O(1),\;\; i=1,\ldots,\xi.$$ 
\end{Theorem}

\begin{Remark}
    Note that here, in contrast to the approach followed by Series in \cite{Series1}, we do not need to exclude from our statements the case of `exceptional curves' and to be dealt with separately. For completeness, we include a definition of exceptional curves, but the interested reader should see \cite{Series1} for a deeper discussion.
\end{Remark}

\begin{Definition}
A geodesic lamination $\eta = \sum_{i = 1}^{k}a_i \delta_{\g_i}$ is \textit{exceptional} if the matrix $\left(q_i(\g_j)\right)_{\substack{i = 1, \ldots, \xi\\ j = 1,\ldots,k}}$ has no maximal rank.
\end{Definition}

We are now ready to prove the theorem.

\begin{proof}[Proof of Theorem \ref{TA}]
    We will use the previous notation, that is we will write $\tau_i(\th) = \tau_i = x_i+i y_i,\;\; \rho = \| (y_{1}, \ldots, y_{\xi})\|$, and $\eta_i= \frac{y_i}{\rho}$, where the dependence on $\th$ is clear. By Theorem~\ref{Prop6_1}, we have $y_i - \frac{4}{\th q_i} = O(1)$. On the other hand, with $t = t(\th)$ as in the statement of the theorem, we find $ \Im w_i(t) = t \frac{q_1}{q_i} = \frac{4}{\th q_i}$. Thus for $i=1,\ldots,\xi$ we have
    $$|\Im \tau_i(\th) - \Im w_i(t(\th))| = O(1),$$ as $\th \to 0$, as we wanted to prove.
    
    Now, let's deal with the coordinates $x_i = \Re \tau_i(\th)$. Given $\g_1,\ldots,\g_k$, let $\g_{k+1},\ldots,\g_{\xi}$, $D_{k+1},\ldots, D_{\xi}$ the curves defined by Lemma \ref{pants}. If $(\tau_1,\ldots,\tau_\xi) \in \P_{\eta}$, then the curves $\g_{1},\ldots, \g_k, \g_{k+1},\ldots,\g_{\xi}$, $D_{k+1},\ldots, D_{\xi}$ are all bending lines of $G(\tau_1,\ldots,\tau_\xi)$. It follows, that $$\Im \Tr(\g_i) = \Im \Tr(D_j) = 0$$ for $i = 1,\ldots, \xi$ and $j = k+1, \ldots, \xi$. So, by Proposition \ref{prop7_1}, it follows that $$E_{\zeta} (\tau_{1}, \ldots, \tau_{\xi}) = O(\th) \;\;\mathrm{as} \;\;\th \to 0$$ for $\zeta \in \{\g_1,\ldots,\g_{\xi}, D_{k+1},\ldots, D_{\xi}\}$. Defining $\eta = \eta_1\cdots\eta_{\xi}$ and regarding these as equations in $\R^{2\xi}$ for a parameter ${\bf u} \in \R^{2\xi}$, where
    \begin{eqnarray*}
    	{\bf u}\; = \;(u_{1 1}, u_{1 2}, \ldots, u_{\xi 1}, u_{\xi 2})\;= \;\eta (\frac{x_1}{\eta_1},\; \frac{1}{\eta_1},\;\ldots,\; \frac{x_{\xi}}{\eta_{\xi}},\; \frac{1}{\eta_{\xi}}), 	
    \end{eqnarray*}
    we have, for $\zeta \in \{\g_1,\ldots,\g_{\xi}, D_{k+1},\ldots, D_{\xi} \}$, 
    \begin{equation}\label{eqn:almostzeta} 
        \i(\zeta) \cdot {\bf u}= O( {\th}).
	\end{equation}
By Theorem \ref{disj}, we have $\Omega_{\rm Th} (\g_i, \zeta) = 0$ for $i = 1,\ldots,k$ for any $\zeta \in \mathrm{lk}_{\CC}(\g_i) \cup \{\g_1, \ldots, \g_k\}$. Hence $\i(\zeta) \cdot \i(\g_i)^{*} = 0$ for $i = 1,\ldots,k$ and for all $\zeta \in \{\g_{1},\ldots,\g_{\xi}, D_{k+1},\ldots, D_{\xi} \}$. Since $\i(\g_1), \ldots, \i(\g_\xi)$, $\i(D_{k+1}) \ldots, \i(D_\xi)$ are independent, it follows that we can write 
	\begin{equation}\label{eqn:almost} 
	    {\bf u} (\th) = \lambda_1 (\th) \i(\gamma_1)^* + \ldots + \lambda_k (\th) \i(\gamma_k)^* + \eta (\th) {\bf v}(\th) 
	\end{equation}
where ${\bf v} = {\bf v}(\th) $ is in the linear span of $\i(\g_1), \ldots, \i(\g_\xi)$, $\i(D_{k+1}) \ldots, \i(D_\xi)$ and $||{\bf v}|| = 1$.

Using \eqref{eqn:almostzeta} we find that ${\bf u} \cdot {\bf v}= O( \th)$ (where the constants depend on $\i(\g_1), \ldots, \i(\g_\xi)$, $\i(D_{k+1}) \ldots, \i(D_\xi)$). Then ${\bf v} \cdot \i(\gamma_i)^* = 0$ for $i = 1,\ldots,k$ gives $\eta (\th) = O(\th)$. Equating the two sides of~\eqref{eqn:almost} gives 
\begin{equation}\label{eqn:almost*} 
	\begin{split}
		& u_{i 1} =  \frac{\eta x_i }{\eta_i} = \lambda_1 p_i(\g_1) + \cdots + \lambda_k p_i(\g_k) +O( \th),\\
		& u_{i 2} = \frac{\eta}{\eta_i} = -\lambda_1 q_i(\g_1) - \cdots -\lambda_k q_i(\g_k) +O( \th). 
	\end{split}
\end{equation}
So we proved ${\bf u}$ belongs to the $k$--dimensional subspace $\Pi$ generated by $\i(\gamma_1)^*, \ldots, \i(\gamma_k)^*$. Now we want to prove ${\bf u}$ is approximately parallel to the vector $\i(\eta)^*$, that is $(\l_1, \ldots,\l_k)$ is proportional to $(a_1,\ldots,a_k)$. To do this, and to avoid the restriction to non exceptional curves, we modify slightly Series' approach. 

By Corollary \ref{Corol6_1}, we have 
\begin{equation}\label{eqn:cor} 
	\begin{split}
		& \left\|\frac{y_i}{y_j}-\frac{a_1 q_j(\g_1) + \cdots + a_k q_j(\g_k)}{a_1 q_i(\g_1) + \cdots + a_k q_i(\g_k)}\right\| = O(\th)\\
		& \left\|\frac{x_j y_i}{y_j x_i}-\frac{a_1 p_j(\g_1) + \cdots + a_k p_j(\g_k)}{a_1 p_i(\g_1) + \cdots + a_k p_i(\g_k)}\right\| = O(\th).
	\end{split}    
\end{equation}

We can now put this information together as:
$$\left\|(\frac{y_i}{y_j}+i\frac{x_j y_i}{y_j x_i})-\frac{a_1 Q_j(\g_1)+ \cdots + a_k Q_j(\g_k)}{a_1 Q_i(\g_1)+ \cdots + a_k Q_i(\g_k)}\right\| = O(\th),$$
where we defined $Q_i(\g) = q_i(\g)+i p_i(\g)$ in order to keep the notation more neat.
Defining new variables $W_i = \lambda_1 (q_i(\g_1)+i p_i(\g_1)) + \cdots + \lambda_k (q_i(\g_k)+ i p_i(\g_k))$, we have, by \eqref{eqn:almost*}, $\Re W_i = -u_{i 2}$ and $\Im W_i = u_{i 1}$. So we have $$\left\|\frac{\Re W_j}{\Re W_i} -\frac{y_i}{y_j}\right\| = O(\th)\;\;\;\;\mathrm{and} \;\;\;\;\left\|\frac{\Im W_j}{\Im W_i} - \frac{x_j y_i}{y_j x_i}\right\| = O(\th).$$ Hence we get
\begin{equation}\label{eqn:cor2} 
	\left\|(\frac{\Re W_j}{\Re W_i}+i\frac{\Im W_j}{\Im W_i})-(\frac{y_i}{y_j}+i\frac{x_j y_i}{y_j x_i})\right\| = O(\th).
\end{equation}

Now using equations \ref{eqn:cor}, \ref{eqn:cor2} and the definition of the variables $W_i$, we get $$\left\|\frac{\l_1 Q_j(\g_1) + \cdots + \l_k Q_j(\g_k)}{\l_1 Q_i(\g_1) + \cdots + \l_k Q_i(\g_k)}-\frac{a_1 Q_j(\g_1) + \cdots + a_k Q_j(\g_k)}{a_1 Q_i(\g_1) + \cdots + a_k Q_i(\g_k)}\right\| = O(\th). $$
Since this is true for all $i,j = 1,\ldots,\xi,\;\;i\neq j$, and since the matrix $(Q_r(\g_s))_{\substack{ r = 1, \ldots, \xi\\
s = 1,\ldots,k}}$ has maximal rank (because, since the curves $\g_1, \ldots,\g_k$ are distinct , the lines of that matrix are linearly independent) and since the norm of the vector $(\l_1-a_1, \ldots, \l_k-a_k)$ is one, then we can conclude the following: $$\left\|\frac{\l_i}{\l_j}-\frac{a_i}{a_j}\right\| = O(\th), \;\;\forall i,j = 1,\ldots,k,\;\;i\neq j,$$ that is ${\bf u} = \alpha \i(\eta)^*$ for some $\alpha > 0$, as we wanted to prove. 
\end{proof}

\begin{Remark}
    We were able to get rid of the hypothesis of non-exceptionality, since we looked simultaneously at both the length and the twist of the Dehn--Thurston coordinates for the distinct curves $\g_1, \ldots,\g_k$.
\end{Remark}

\bibliographystyle{amsplain}

\end{document}